\documentclass[9pt]{scrartcl}
\usepackage{ifthen}
\usepackage{geometry}
\newcommand{\ifmma}[1]{{}}
\newcommand{\ifmtwoan}[1]{{#1}}

\provideboolean{tableOfContents}
\setboolean{tableOfContents}{true}

\usepackage{amsmath,amsfonts,amssymb,amsthm}
\usepackage{graphicx,caption}
\setkeys{Gin}{draft=false}
\usepackage{psfrag}
\usepackage{color}
\usepackage[hang,loose]{subfigure}
\usepackage{comment}
\usepackage{nicefrac}
\usepackage{bm}
\usepackage[breaklinks=true,colorlinks=true,urlcolor=black,linkcolor=black]{hyperref}
\usepackage{xspace}
\usepackage[normalem]{ulem} 
\usepackage{pgfplots,pgfplotstable}
\pgfplotsset{compat=1.5.1}
\usepackage{siunitx}
\usetikzlibrary{plotmarks}
\usepackage{rotating}
\usepackage{marginnote}
\RequirePackage{nicefrac}
\usepackage{ifthen}
\usepackage{pdflscape}
\usepackage{enumerate}
\usetikzlibrary{arrows}
\pgfkeys{/pgf/images/include external/.code=\includegraphics{#1}}

\newtheorem{theorem}{Theorem}[section]
\newtheorem{lemma}[theorem]{Lemma}

\newtheorem{definition}[theorem]{Definition}

\newtheorem{remark}[theorem]{Remark}

\newcommand{\eps}{\varepsilon} 

\newcommand{\IR}{{\mathbb R}}
\newcommand{\IC}{{\mathbb C}}
\newcommand{\IN}{{\mathbb N}}

\newcommand{\vf}{{\mathbf f}}
\newcommand{\vgg}{{\mathbf g}}

\newcommand{\vn}{{\mathbf n}}

\newcommand{\vx}{{\mathbf x}}

\newcommand{\vv}{{\mathbf v}}
\newcommand{\vw}{{\mathbf w}}

\DeclareMathOperator{\imagpart}{Im}
\renewcommand{\Im}{{\imagpart}}

\newcommand{\imag}{\ensuremath{{\rm i}}}

\newcommand{\wh}{\widehat}
\newcommand{\wt}{\widetilde}

\DeclareMathOperator{\Div}{div}

\newcommand{\plcurl}{\operatorname{{\bf curl}}_{2D}}
\newcommand{\scurl}{\operatorname{curl}_{2D}}

\newcommand{\supp}{\operatorname{supp}}

\newcommand{\zerobf}{\boldsymbol{0}}

\newcommand{\dsurf}{\mathrm{d}\sigma(\vx)}

\newcommand{\dx}{\mathrm{d}\vx}

\newcommand{\dS}{\mathrm{d}S}

\def\half{\frac{1}{2}}

\newcommand{\eg}{\textit{e.\,g\mbox{.}}\xspace}

\newcommand{\ie}{\textit{i.\,e\mbox{.}}\xspace}
\newcommand{\cf}{\textit{cf.} }

\newcommand{\appr}{{\mathrm{appr}}}
\newcommand{\apprA}[2]{#1_{\appr,#2}}

\tikzset{dash dotted/.style={dash pattern=on 1pt off 4pt on 6pt off 4pt}} %

\numberwithin{equation}{section}

\definecolor{green07}{rgb}{0, 0.7, 0}
\definecolor{green08}{rgb}{0, 0.8, 0}

\graphicspath{{}}

\begin{document}
    
    \title{Impedance boundary conditions for acoustic time harmonic wave propagation
        in viscous gases}
    \author{Kersten Schmidt and Anastasia Th\"ons-Zueva}%

    %
    %
    \date{Version: \today}

    \maketitle
    
        \begin{abstract}
        We present Helmholtz or Helmholtz like equations for the
        approximation of the time-harmonic wave propagation in gases with
        small viscosity, which are completed with local boundary
        conditions on rigid walls. %
        We derived approximative models based on the method of multiple
        scales for the pressure and the velocity separately, both up to
        order 2.
        Approximations to the pressure are described by the Helmholtz
        equations with impedance boundary conditions, which relate its
        normal derivative to the pressure itself. %
        The boundary conditions from first order on are of Wentzell type
        and include a second tangential derivative of the pressure
        proportional to the square root of the viscosity, and take thereby
        absorption inside the viscosity boundary layer of the underlying
        velocity into account.
        
        The velocity approximations are described by Helmholtz like
        equations for the velocity, where the Laplace operator is replaced
        by $\nabla \Div$, and the local boundary conditions relate the
        normal velocity component to its divergence. The velocity
        approximations are for the so-called far field and do not exhibit
        a boundary layer. %
        Including a boundary corrector, the so called near field, the
        velocity approximation is accurate even up to the domain boundary.
        
        The boundary conditions are stable and asymptotically exact, which
        is justified by a complete mathematical analysis. %
        The results of some numerical experiments are presented to
        illustrate the theoretical foundation.
    \end{abstract}
    
    \numberwithin{equation}{section}
    
    \ifthenelse{\boolean{tableOfContents}}{%
        \tableofcontents
    }{}%
    
    \section{Introduction}
    
    In this study we are investigating the acoustic equations as a
    perturbation of the Navier-Stokes equations around a stagnant uniform
    fluid, with mean density $\rho_0$ and without heat flux.  For gases
    the (dynamic) viscosity $\eta$ is very small and leads to
    \textit{viscosity boundary layers} close to walls.
    To resolve the boundary layers with (quasi-)uniform meshes, the mesh
    size has to be of the same order, which leads to very large linear
    systems to be solved. %
    This is especially the case for the very small boundary layers of
    acoustic waves. %
    In its turn, the pressure field does not possess a boundary layer,
    however, this fact cannot be used without some preliminary adjustments
    as there are no existing physical boundary conditions for pressure.
    
    In an earlier work~\cite{Schmidt.Thoens.Joly:2014} we derived a
    complete asymptotic expansion for the problem based on the technique
    of multiscale expansion in powers of $\sqrt{\eta}$ which takes into
    account curvature effects.  This asymptotic expansion was rigorously
    justified with optimal error estimates.
    In this article we propose and justify, based on the asymptotic
    expansion in~\cite{Schmidt.Thoens.Joly:2014}, (effective)
    \textit{impedance boundary conditions} for the velocity as well as
    the pressure for possibly curved boundaries. %
    Similar techniques to derive approximative models have been used for
    thin
    sheets~\cite{Schmidt.Tordeux:2011,Delourme.Haddar.Joly:2012,Perrussel.Poignard:2013}
    or for conducting bodies~\cite{Haddar.Joly.Nguyen:2005}.
    The advantage of using this approach lies in the fact that the
    solution can be divided into the far field with specified boundary
    condition, \ie, impedance boundary condition, and a correcting near
    field, which helps to avoid resolving the boundary layer.
    A similar strategy is used for deriving a wall boundary conditions
    for acoustic plane waves in presence of a shear flow~\cite{Auregan.Starobinski.Pagneux:2001}. 
    
    The article is subdivided as follows. In Sec.~\ref{sec:ModelDef} we
    define the model problem of the viscous acoustic equations for
    velocity and pressure and state the impedance boundary conditions for
    the velocity and for the pressure as well as the stability and
    modelling error estimates. %
    Sec.~\ref{sec:ImpedanceBC} is dedicated to the derivation of the
    impedance boundary condition on the basis of the asymptotic expansion
    presented in~\cite{Schmidt.Thoens.Joly:2014}. %
    The well-posedness as well as estimates of the modelling error of the
    approximative models with the impedance boundary conditions will be
    shown in Sec.~\ref{sec:justification}.  Results of some numerical
    experiments in Sec.~\ref{sec:NumResults} shall emphasize the validity
    of the theoretical findings.
    
    \section {Model problem definition and main results}
    \label{sec:ModelDef}
    
    \subsection{Geometry and model problem}
    
    Let $\Omega \subset \IR^2$ be a bounded Lipschitz domain with boundary
    $\partial\Omega$ , where $\vn$ denotes the outer normalised normal vector. %
    If $\partial\Omega$ is piecwise $C^2$ then $\kappa$ denotes the (signed) curvature \emph{a.e.} on $\partial\Omega$
    which is positive on convex parts of $\partial\Omega$.
    
    \begin{figure}[t]
        \centering
        \subfigure[]
        {
            \label{fig:coordinate}
            \begin{tikzpicture}[scale=2.5,allow upside down,
            interface/.style={
                postaction={draw,decorate,decoration={border,angle=-45,
                        amplitude=0.3cm,segment length=2mm}}},
            ]
            \draw [line width=1pt,interface] plot [smooth cycle] coordinates 
            {(0,0) (0.5,-0.7) (1.3,-0.6) (2.2,-0.7) (2.8,0) (2.4,1.2) (0.8,1)}
            node at (2.2,-0.2) {$\Omega$}
            node at (2.1,1.16) {$\partial \Omega$}
            node at (1.3,-0.6) [sloped,inner sep=0cm,above,anchor=south west,
            minimum height=1cm,minimum width=1cm](N){}
            node at (0,0) [sloped,inner sep=0cm,above,anchor=north east,
            minimum height=1cm,minimum width=1cm](N2){};
            \path (N.south west)%
            edge[-stealth',blue,thick] node[left] {$ s$} (N.north west)
            edge[-stealth',blue,thick] node[above] {$ t$} (N.south east);
            \path (N2.north west)%
            edge[stealth-,blue,line width=1pt] node[above] {$\vn(t)$} (N2.north east);
            \foreach \ll in {1, ..., 6} { 
                \draw (1.5,0.3) circle (\ll/20);
            };
            \draw (1.2,0.5) node {$\vf$};
            \draw [color = white] (3,0) rectangle (3.5,1);
            \end{tikzpicture} 
        }
        \subfigure[]{
            \label{fig:torus}
            \begin{tikzpicture}[scale=3,allow upside down,
            interface/.style={
                postaction={draw,decorate,decoration={border,angle=45,
                        amplitude=0.2cm,segment length=1.5mm}}},
            ]
            \draw [line width=1pt,interface] (1,0) -- (0,0);
            \draw [line width=1pt,interface] (0,2) -- (1,2);
            \draw [line width=1pt] (1,0) -- (1,2);
            \draw [line width=1pt] (0,2) -- (0,0);
            \draw [line width=1pt,interface] (0.25,1.5) circle (0.15);
            \foreach \ll in {1, ..., 8} { 
                \draw (0.75,0.5) circle (\ll/35);
            };
            \draw (0.3,0.07) node {$\partial \Omega$};
            \draw (0.7,1.93) node {$\partial \Omega$};
            \draw (0.35,1.7) node {$\partial \Omega$};
            \draw (0.5,1) node {$\Omega$};
            \draw (0.47,0.5) node {$\vf$};
            \end{tikzpicture}
        }
        \caption{(a) Definition of a general domain with a local coordinate system $(t,s)$ 
            close to the wall; (b) Definition of a torus domain for numerical simulations.}
    \end{figure}
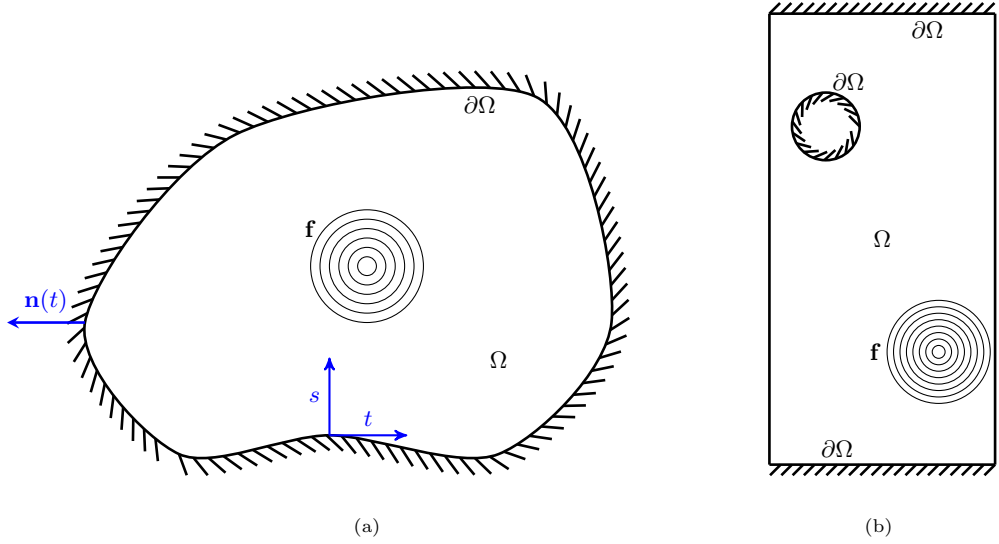
    
    We consider the time-harmonic acoustic velocity $\vv$ and acoustic
    pressure $p$ (the time regime is $\mathrm{e}^{-\imag\omega t}$,
    $\omega \in \IR^+$) which are described by the coupled system
    \begin{subequations}
        \label{eq:navier.stokes}
        \begin{align}
        -\imag\omega\rho_0  \vv  +\nabla p - \eta\Delta \vv -\eta'\,\nabla\Div \vv %
        &= \vf,%
        \hspace{4em}\text{in }\Omega,
        \label{eq:navier.stokes:mom} \\ 
        -\imag\omega p + \rho_0 c^2\,\Div \vv &= 0,%
        \hspace{4em}\text{in }\Omega, \label{eq:navier.stokes:pr}\\
        \label{eq:navier.stokes:bound}
        \vv &= \zerobf, \hspace{4em} \text{on } \partial\Omega.
        \end{align}
    \end{subequations}
    
    In the \textit{momentum equation}~\eqref{eq:navier.stokes:mom} with
    some known source term $\vf$ the viscous dissipation in the momentum
    is not neglected as we consider near wall regions. %
    Here, $\rho_0$ is the density of the media, $c$ is the speed of sound,
    $\eta > 0$ is the dynamic viscosity and %
    $\eta'$ the second (volume) viscosity. %
    {Both shall take small values and we call $\gamma' = \eta'/\eta$
        their quotient.}
    The system is completed by \textit{no-slip} boundary conditions. %
    Similar acoustic equations have been derived and studied 
    in~\cite{Howe:1998,Rienstra:2010,Landau:1959} for a stagnant flow
    and in~\cite{Auregan.Starobinski.Pagneux:2001, Munz:2007, Howe:1998, Howe:1979, 
        Riensra.Darau:2011} for the case that a mean flow is present.
    
    It is well-known that the acoustic velocity field exhibits a
    boundary layer of thickness $O(\sqrt{\eta})$ starting at the rigid
    wall, see \eg~\cite{Auregan.Starobinski.Pagneux:2001, Andreev.Kopteva:2008, 
        Iftimie.Sueur:2010, Schmidt.Thoens.Joly:2014} and the references there. In the following we
    propose definitions of far field velocities, which approximate the
    acoustic velocity outside the boundary layer, correcting near field
    velocities and approximative pressure distributions.
    
    \subsection{Defintions of the approximative models with impedance boundary conditions}
    
    In this section we present approximative models for the far field
    pressure $\apprA{p}{N}$ of order $0$, $1$ and $2$ (Sec.~\ref{sec:ModelDef:pressure})
    and for the far field velocity $\apprA{\vv}{N}$  of order $0$, $1$ and $2$ (Sec.~\ref{sec:ModelDef:velocity}), which include in particular
    impedance boundary conditions. %
    For both kind of approximative models the approximations to the respective other quantity, acoustic velocity $\apprA{\vv}{N}$ or pressure $\apprA{p}{N}$, results in a post-processing step
    by algebraic equations from $\apprA{p}{N}$ or $\apprA{\vv}{N}$, respectively. 
    Moreover, velocity boundary layer correctors $\apprA{\vv^{BL}}{N}$ 
    can be computed from the far field velocity. %
    They are derived for smooth boundaries, but their (weak) formulations can be defined if the domain $\Omega$ is Lipschitz, and piecewise $C^2$ boundary 
    is required for the models of order $2$ that include the curvature.

    \subsubsection{Approximative models for the pressure}
    \label{sec:ModelDef:pressure}
    
    \noindent
    The approximative model of order~0 is given by
    \begin{subequations}
        \label{eq:pappr:0}
        \begin{align}
        \label{eq:pappr:0:PDE}
        \Delta \apprA{p}{0} + \frac{\omega^2}{c^2} \apprA{p}{0} &= \Div \vf\ ,  &&\text{in }\Omega\ ,\\
        \label{eq:pappr:0:bc}
        \nabla \apprA{p}{0} \cdot \vn  &= \vf \cdot \vn\ ,&&\text{on }\partial\Omega\ .
        \end{align}
    \end{subequations}
    If the source $\vf$ is localized away from the boundary $\partial\Omega$ then the boundary conditions are homogeneous, likewise the following impedance conditions of higher order.
    We define a model of order~1 by
    \vspace{-0.3em}
    \begin{subequations}
        \label{eq:pappr:1}
        \begin{align}
        \label{eq:pappr:1:PDE}
        \Delta \apprA{p}{1} + \frac{\omega^2}{c^2} \apprA{p}{1} &= \Div \vf\ ,  &&\text{in }\Omega\ ,  \\
        \label{eq:pappr:1:bc}
        \nabla \apprA{p}{1} \cdot \vn + (1+\imag)\sqrt{\frac{\eta}{2\omega\rho_0}}\partial_\Gamma^2 
        \apprA{p}{1} 
        &= \vf \cdot \vn-(1+\imag)\sqrt{\frac{\eta}{2\omega\rho_0}} \partial_\Gamma(\vf\cdot\vn^\bot)\ ,&&\text{on }\partial\Omega\
        \end{align}
    \end{subequations}
    with the tangential derivative $\partial_\Gamma$ (see below)
    and for order~2 we define
    \begin{subequations}
        \label{eq:pappr:2}
        \begin{align}
        \label{eq:pappr:2:PDE}
        \Big(1- \frac{\imag\omega(\eta+\eta')}{\rho_0 c^2}\Big)\Delta \apprA{p}{2} +
        \frac{\omega^2}{c^2}\apprA{p}{2} &= \Div \vf\, \qquad \text{in }\Omega\ , \\
        \label{eq:pappr:2:bc}
        \Big(1- \frac{\imag\omega(\eta+\eta')}{\rho_0 c^2}\Big)\nabla \apprA{p}{2} \cdot \vn + (1+\imag)\sqrt{\frac{\eta}{2\omega\rho_0}}
        \partial_\Gamma^2 \apprA{p}{2} 
        +\frac{\imag\eta}{2\omega\rho_0}\partial_\Gamma(\kappa\partial_\Gamma \apprA{p}{2}) &= \\
        \vf \cdot \vn -(1+\imag)\sqrt{\frac{\eta}{2\omega\rho_0}} \partial_\Gamma(\vf\cdot\vn^\bot)
        - \frac{\imag\eta}{2\omega\rho_0} \partial_\Gamma(\kappa \vf\cdot\vn^\bot)
        &- \frac{\imag\eta}{\omega\rho_0}\plcurl\scurl\vf\cdot\vn\ ,\quad\text{on }\partial\Omega\ .\nonumber
        \end{align}
    \end{subequations}
    
    The weak formulation for~\eqref{eq:pappr:0} reads: Seek $\apprA{p}{N} \in H^1(\Omega)$
    such that for all $q' \in H^1(\Omega)$
    \begin{align}
    \label{eq:pappr:0:var}
    \int_\Omega \nabla \apprA{p}{N}\cdot \nabla q' - \frac{\omega^2}{c^2} \apprA{p}{N}q\, \dx
    &= \int_\Omega \vf\cdot \nabla q' \dx\ .
    \end{align}
    The impedance boundary conditions~\eqref{eq:pappr:1:bc}
    and~\eqref{eq:pappr:2:bc} are of Wentzell type, %
    see \cite{BonnaillieNoel.Dambrine.Herau.Vial:2010,Schmidt.Heier:2013}
    for the functional framework. %
    With the Sobolev space $H^1(\Omega) \cap H^1(\partial\Omega)$ with functions that are in $H^1(\Omega)$ and whose traces are in $H^1(\partial\Omega)$
    the weak formulations for the systems~\eqref{eq:pappr:1} and~\eqref{eq:pappr:2} 
    are given as: Seek $\apprA{p}{N} := H^1(\Omega) \cap H^1(\partial\Omega)$
    such that for all $q' \in H^1(\Omega) \cap H^1(\partial\Omega)$
    \begin{multline}
    \int_\Omega \Big(1- \frac{\imag\omega(\eta+\eta')\delta_{N=2}}{\rho_0 c^2}\Big) \nabla\apprA{p}{N}\cdot \nabla q' - \frac{\omega^2}{c^2} \apprA{p}{N}q'\, \dx
    - \int_{\partial\Omega} \left((1 + \imag)\sqrt{\frac{\eta}{2\omega\rho_0}} + \frac{\imag\eta\delta_{N=2}}{2\omega\rho_0}\kappa\right) \partial_\Gamma \apprA{p}{N} \partial_\Gamma q'\,\dsurf\\
    = \int_\Omega \vf\cdot \nabla q' \dx %
    - \int_{\partial\Omega} \left((1 + \imag)\sqrt{\frac{\eta}{2\omega\rho_0}} + \frac{\imag\eta\delta_{N=2}}{2\omega\rho_0}\kappa\right) \vf\cdot\vn^\bot \partial_\Gamma q' 
    + \frac{\imag\eta\delta_{N=2}}{\omega\rho_0}\plcurl\scurl\vf\cdot\vn q'\,\dsurf\ .
    \end{multline}
    When the far field pressure is computed we may obtain {\em
        a-posteriori} the far field velocity of order 0, 1 and 2 by %
    \begin{subequations}
        \label{eq:pN1:vN}
        \begin{align}  
        \label{eq:pj:v:01}
        \apprA{\vv}{N} &= \frac{\imag}{\rho_0\omega}(\vf - \nabla \apprA{p}{N}),
        \text{ for } N=0, 1,
        && \text{ in }\Omega, \\
        \label{eq:pj:v:2}
        \apprA{\vv}{2} &= \frac{\imag}{\rho_0\omega} \vf - \frac{\imag}{\rho_0\omega} \Big(1- \frac{\imag\omega(\eta+\eta')}{\rho_0 c^2}\Big) \nabla \apprA{p}{2} 
        + \frac{\eta}{\rho_0^2\omega^2}\plcurl\scurl\vf,
        && \text{ in }\Omega,
        \end{align}
    \end{subequations}
    Close to the wall the far field velocities have to be corrected by a
    function 
    \begin{align}
    \label{eq:vappr:BL}
    \apprA{\vw^{BL}}{N} = \sqrt{\frac{2\,\eta}{\omega\rho_0}}\plcurl(\apprA{\phi}{N}\chi),  
    \end{align}
    where $\chi$ is an admissible cut-off function (see Definition~\ref{lem:chi}). %
    
    For the definition of the approximative boundary layer functions we need to introduce a local coordinate system $(t,s)$
    where points close to $\partial\Omega$ can be uniquely written as
    \begin{align}
    \label{eq:localCoord}
    \vx(t,s) = \vx_{\partial\Omega}(t) - s \vn(t)
    \end{align}
    where the boundary is described by the mapping $\vx_{\partial\Omega}(t)$ from an interval $T \in \IR$ and $s$ is the distance from the boundary (see Fig.~\ref{fig:coordinate}).
    Without loss of generality we can assume $|\vx_{\partial\Omega}'(t)| = 1$ for all $t \in T$ and the tangential derivative is given by $\partial_\Gamma v(\vx) = \partial_t v(\vx(t,s))$. %
    Then, 
    
    $\apprA{\phi}{N}(\vx) = \apprA{\wt{\phi}}{N}(t,s\sqrt{\frac{\omega\rho_0}{2\eta}})$ %
    in the so-called $\chi$-neighbourhood of the boundary and
    \begin{align}
    \label{eq:vappr:phi}
    \apprA{\wt{\phi}}{N}(t,S) &:= %
    \half(1+\imag)\,\mathrm{e}^{-(1-\imag) S}\sum_{\ell=0}^N \left(\frac{2\,\eta}{\omega\rho_0}\right)^{\frac{\ell}{2}} E_\ell(\apprA{\vv}{N}\cdot\vn^\bot)(t,S).
    \end{align}
    Here, the operators $E_\ell$, which were recursively defined
    in~\cite[Lemma~A.1]{Schmidt.Thoens.Joly:2014} {(the parameter $\eta_0
        = \omega\rho_0/2$ has to be used in their definition), } will be
    given for $\ell=0,1,2$ in~\eqref{eq:El}. %
    Note, that the $\Div$-free near field correctors~\eqref{eq:vappr:BL}
    can be replaced by $\sqrt{2\eta/(\omega\rho_0)}\chi\plcurl(\apprA{\phi}{N})$ without changing the asymptotic behaviour.
    
    \subsubsection{Approximative models for the acoustic velocity}
    \label{sec:ModelDef:velocity}
    
    \noindent
    The approximative model of order~0 is given by
    \vspace{-0.3em}
    \begin{subequations}
        \label{eq:vappr:0}
        \begin{align}
        \label{eq:vappr:0:PDE}
        \nabla\Div \apprA{\vv}{0} + \frac{\omega^2}{c^2} \apprA{\vv}{0} &= 
        \frac{\imag\omega}{\rho_0 c^2} \vf, && \text{ in }\Omega,\\[0.2em]
        \label{eq:vappr:0:bc}
        \apprA{\vv}{0}\cdot\vn &= 0,&& \text{ on }\partial\Omega,
        \end{align}
    \end{subequations}%
    that of order~1 by
    \vspace{-0.5em}
    \begin{subequations}
        \label{eq:vappr:1}
        \begin{align}
        \label{eq:vappr:1:PDE}
        \nabla\Div \apprA{\vv}{1} + \frac{\omega^2}{c^2} \apprA{\vv}{1} &= 
        \frac{\imag\omega}{\rho_0 c^2} \vf, && \text{ in }\Omega,\\[0.2em]
        \label{eq:vappr:1:bc}
        \apprA{\vv}{1}\cdot\vn - 
        (1+\imag)\frac{c^2}{\omega^2}
        \sqrt{\frac{\eta}{2\omega\rho_0}}
        \partial_\Gamma^2\Div\apprA{\vv}{1}%
        &= \frac{(\imag-1)}{\omega \rho_0} \sqrt{\frac{\eta}{2\omega\rho_0}}
        \partial_\Gamma(\vf\cdot\vn^\bot),&& \text{ on }\partial\Omega,
        \end{align}
    \end{subequations}
    and 
    \begin{subequations}
        \label{eq:vappr:2}
        \begin{align}
        \label{eq:vappr:2:PDE}
        \left(1-\frac{\imag\omega(\eta+\eta')}{\rho_0c^2}\right)\nabla \Div \apprA{\vv}{2} + 
        \frac{\omega^2}{c^2} \apprA{\vv}{2} 
        &= \frac{\imag\omega}{\rho_0 c^2} \vf
        +\frac{\eta}{\rho_0^2c^2}\plcurl\scurl \vf, && \text{ in }\Omega,\\[0.2em]
        \label{eq:vappr:2:bc}
        \apprA{\vv}{2}\cdot\vn - 
        \frac{c^2}{\omega^2}f
        \Big(
        (1+\imag)\sqrt{\frac{\eta}{2\omega\rho_0}}
        \partial_\Gamma^2\Div\apprA{\vv}{2}%
        &+ \frac{\imag\eta}{2\omega\rho_0}
        \partial_\Gamma(\kappa\partial_\Gamma\Div \apprA{\vv}{2})\Big)\nonumber \\
        &= \frac{(\imag-1)}{\omega \rho_0} \sqrt{\frac{\eta}{2\omega\rho_0}}
        \partial_\Gamma(\vf\cdot\vn^\bot)
        - \frac{\eta}{2\omega^2\rho_0^2}\partial_\Gamma(\kappa\,\vf\cdot\vn^\bot),&& \text{ on }\partial\Omega, 
        %
        \end{align}
    \end{subequations}
    defines the approximative model of order~2. %
    
    The impedance boundary conditions~\eqref{eq:vappr:1:bc}
    and~\eqref{eq:vappr:2:bc} have similarities with Wentzell's boundary
    conditions~\cite{Feller:1952, Venttsel:1956, Feller:1957,
        Venttsel:1959}, where, however, the second tangential derivative applies
    to the Neumann trace $\Div\apprA{\vv}{N}$, and not to the Dirichlet
    trace, which is here $\apprA{\vv}{N}\cdot\vn$. %
    The limit velocity model and the approximative models of higher order are of different kind as the exact model~\eqref{eq:navier.stokes}
    since the $\Delta\apprA{\vv}{N}$ and so $\plcurl\scurl\apprA{\vv}{N}$ are missing and there is no condition on the tangential component.
    
    The weak formulation for~\eqref{eq:vappr:0} reads: Seek $\apprA{\vv}{0} \in H_0(\Div,\Omega)$ such that
    \begin{align}
    \label{eq:vappr:var:0}
    \int_{\Omega} \Div \apprA{\vv}{0}\Div \vv' - \frac{\omega^2}{c^2}\apprA{\vv}{0}\cdot\vv'\dx &= \int_\Omega \vf\cdot\vv'\dx\quad\forall \vv'\in H_0(\Div,\Omega).
    \end{align}
    Introducing the Lagrange multipliers $\apprA{\lambda}{N} = \left(1-\frac{\imag\omega(\eta+\eta')\delta_{N=2}}{\rho_0c^2}\right)\Div\apprA{\vv}{N}$, $N = 1,2$ 
    on $\partial\Omega$ we find the mixed variational formulations for the systems~\eqref{eq:vappr:1} and~\eqref{eq:vappr:2}:  Seek 
    $(\apprA{\vv}{N},\apprA{\lambda}{N}) \in H(\Div,\Omega)\times H^1(\partial\Omega)$ such that
    for all $(\vv',\lambda') \in H(\Div,\Omega)\times H^1(\partial\Omega)$ 
    \begin{subequations}
        \label{eq:vappr:var:12}
        \begin{align}
        \nonumber
        \int_{\Omega} \Big(1 - \frac{\imag\omega(\eta+\eta')\delta_{N=2}}{\rho_0c^2}\Big)\Div \apprA{\vv}{N}\Div \vv' - \frac{\omega^2}{c^2}\apprA{\vv}{N}\cdot\vv'\,\dx\quad \\
        - \int_{\partial\Omega} \apprA{\lambda}{N}\,\vv'\cdot\vn\,\dsurf %
        &= \int_\Omega \Big(\vf + \frac{\eta\,\delta_{N=2}}{\rho_0^2c^2}\plcurl\scurl\vf\Big)\cdot\vv'\,\dx,
        \label{eq:vappr:var:12:1}\\
        \label{eq:vappr:var:12:2}
        \int_{\partial\Omega}\apprA{\vv}{N}\cdot\vn\, \lambda' 
        + \frac{c^2}{\omega^2}\frac{(1+\imag)\sqrt{\frac{\eta}{2\omega\rho_0}} + \frac{\imag\eta\,\delta_{N=2}}{2\omega\rho_0}\kappa}{1-\frac{\imag\omega(\eta+\eta')\delta_{N=2}}{\rho_0c^2}}\;\partial_\Gamma \apprA{\lambda}{N} \partial_\Gamma \lambda'  
        \,\dsurf
        &= \int_{\partial\Omega} \Big(\frac{1-\imag}{\omega\rho_0}\sqrt{\frac{\eta}{2\omega\rho_0}}
        + \frac{\eta\,\delta_{N=2}}{2\omega^2\rho_0^2}\kappa\Big)
        \vf\cdot\vn^\bot \partial_\Gamma\lambda'\,\dsurf.
        \end{align}
    \end{subequations}
    
    When the far field velocity is computed, we may obtain {\em
        a-posteriori} the far field pressure in $\Omega$ of order 0, 1 and 2
    by %
    \begin{align}
    \label{eq:vN1:pN}
    \apprA{p}{N} &= -\frac{\imag \rho_0 c^2}{\omega}\,\Div \apprA{\vv}{N}\ .
    \end{align}
    Moreover, the near field velocity $\vv_{\appr,N}^{BL}$ is then given by~\eqref{eq:vappr:BL} and~\eqref{eq:vappr:phi}.
    
    \subsection{Well-posedness and modelling error}
    
    Obviously, the system for the limit pressure~\eqref{eq:pappr:0} has no unique solutions for frequencies $\omega > 0$ for that $\frac{\omega^2}{c^2}$ is an eigenvalue of $-\Delta$ -- the eigenfrequencies. 
    In~\cite{Schmidt.Thoens.Joly:2014} we have shown that the limit velocity system~\eqref{eq:vappr:0} has eigensolutions for the same frequencies, for which it does not provide a unique solution. %
    If $\omega$ takes such a value by the Fredholm alternative~\cite{Sauter.Schwab:2011}, the systems provides solutions if the source is orthogonal to all eigenfunctions. %
    This is, however, in practise rather unlikely.
    The additional dissipative term in the pressure and velocity systems of order $1$ are not enough to guarantee uniqueness for all frequencies in general. There might be eigenfunctions 
    of the pressure limit systems that do not vary on $\partial\Omega$ such that they satisfy the first order pressure system~\eqref{eq:pappr:1} with $\vf=\zerobf$. %
    Also eigenfunction of the velocity limit systems whose divergence on $\partial\Omega$, \ie, the Neumann trace, do not vary are eigenfunctions of 
    the first order velocity system~\eqref{eq:vappr:1}. %
    Only the volumic dissipative term of the two systems of order $2$ guarantee, as for the original model, for existence and uniqueness for all frequencies $\omega > 0$. %
    These properties will be shown and discussed by numerical experiments in Sec.~\ref{sec:NumResults}. %
    However, in the analysis we assume that $\omega$ is not an eigenfrequency of the limit system.
    
    \begin{theorem}[Stability, existence and uniqueness of $(\apprA{\vv}{N}, \apprA{p}{N})$]
        Let $\Omega$ be an open Lipschitz domain whose boundary is piecewise $C^2$ for $N = 2$, let $\frac{\omega^2}{c^2}$ be distinct from the Neumann eigenvalues
        of $-\Delta$ of $\Omega$ and let 
        $\vf \in H(\scurl,\Omega)$ and $\plcurl\scurl\vf \in H(\scurl,\Omega)$ for $N = 2$.
        Then, there exists a constant $\eta_0 > 0$ such that for all $\eta \in (0,\eta_0)$
        the systems~\eqref{eq:pappr:0}--\eqref{eq:pappr:2} provide unique solutions $\apprA{p}{N}$, $N=0,1,2$
        and the systems~\eqref{eq:vappr:0}--\eqref{eq:vappr:2} provide unique solutions $\apprA{\vv}{N}$, $N=0,1,2$, respectively.
        Furthermore, there exists a constant $C$ independent of $\eta$ such that 
        the stability estimates 
        \begin{subequations}
            \label{eq:pappr:stability}    
            \begin{align}
            \|\apprA{\vv}{N}\|_{(H(\Div,\Omega)} + \|\apprA{p}{N}\|_{H^1(\Omega)} %
            &\leq C\, \Big(\|\vf\|_{L^2(\Omega)} 
            + \eta\,\delta_{N=2} \|\plcurl\scurl\vf\|_{L^2(\Omega)}\Big)\ ,\\     
            \|\scurl\apprA{\vv}{N}\|_{L^2(\Omega)} 
            &\leq C\, \Big(\|\scurl \vf\|_{L^2(\Omega)} 
            + \eta\,\delta_{N=2} \|\scurl \plcurl\scurl\vf\|_{L^2(\Omega)}\Big)
            \end{align}
        \end{subequations}
        holds. Moreover, the approximative models are equivalent as the identities~\eqref{eq:pN1:vN} and~\eqref{eq:vN1:pN} hold.
        \label{lem:stability}
    \end{theorem}

    \begin{remark}
        The equivalent systems~\eqref{eq:pappr:2}--\eqref{eq:pj:v:2} and ~\eqref{eq:vappr:2}--\eqref{eq:vN1:pN} provide
        unique solution $(\apprA{\vv}{2},\apprA{p}{2}) \in (H^1(\Omega))^2
        \times H^1(\Omega)$ for any $\omega>0$, however, with a constant $C = C(\eta)$ that may blow up for $\eta\to0$. 
    \end{remark}
    
    \begin{theorem}[Modelling error]
        \label{lem:error}
        Let $\Omega$ be an open smooth domain, let $\frac{\omega^2}{c^2}$ be distinct from the Neumann eigenvalues
        and let $\vf \in (L^2(\Omega))^2$ where $\scurl\vf \in H^m(\Omega)$ for any $m \in \IN$
        and $\vf \in H^m(\wt{\Omega}))^2$ for any $m \in \IN$ in some neighbourhood $\wt{\Omega} \subset
        \Omega$ of $\partial\Omega$, \ie, $\partial\Omega \subset
        \partial\wt{\Omega}$.
        Then,
        the approximative solution
        $(\apprA{\vv}{N}, \apprA{p}{N})$ %
        for $N = 0, 1, 2$ satisfies
        \begin{subequations}
            \ifmtwoan{\begin{align}
                \label{lem:error:1}
                \|p - \apprA{p}{N}\|_{H^1(\Omega)}
                \leq C\, \eta^{\frac{N+1}{2}},
                \end{align} %
                \ifmma{\begin{align}
                    \label{lem:error:1}
                    \|p - \apprA{p}{N}\|_{H^1(\Omega)}
                    \leq C\, \eta^{\frac{N+1}{2}},
                    \end{align}}} %
            and for any $\delta > 0$
            \begin{align}
            \label{lem:error:2}
            \|\vv - \apprA{\vv}{N}\|_{(H^1(\Omega\setminus\overline{\Omega}_\delta))^2} 
            &\leq C_{\delta,N}\, \eta^{\frac{N+1}{2}},
            \end{align}
        \end{subequations} %
        where $\Omega_\delta$ is the original domain without a
        $\delta$-neighbourhood of $\partial\Omega$ and where the constants
        $C$, $C_{\delta,N} > 0$ do not depend on $\eta$. %
        %
    \end{theorem}

    The proofs will be given in Sec.~\ref{sec:justification}.

    
    \section{Derivation of impedance boundary conditions}
    \label{sec:ImpedanceBC}
    
    
    \subsection{Equations for asymptotically small viscosity}
    
    To investigate the solution of~\eqref{eq:navier.stokes} for small
    viscosities, we introduce a small dimensionless parameter $\eps \ll
    1$, $\eps \in \IR^+$ and replace $\eta, \eta'$ by
    {$\eps^2\omega\rho_0/2$, $\eps^2\gamma'\omega\rho_0/2$ (corresponding
        to $\eta_0 = \omega\rho_0/2$, $\eta_0' = \gamma'\,\omega\rho_0/2$
        in~\cite{Schmidt.Thoens.Joly:2014}), respectively.}  %
    In this way the
    boundary layer thickness will become proportional to $\eps$. %
    The solution of~\eqref{eq:navier.stokes}, 
    respectively, will be labelled $\vv^\eps$ and $p^\eps$ due to its
    dependence on~$\eps$, \ie,
    \begin{subequations}
        \label{eq:navier.stokes:eps}
        \begin{align}
        -\imag\omega\rho_0 \vv^\eps  +\nabla p^\eps - {\eps^2\frac{\omega\rho_0}{2}\Delta \vv^\eps -\eps^2\frac{\gamma'\,\omega\rho_0}{2}
            \nabla \Div \vv^\eps} &= \vf,%
        \hspace{4em}\text{in }\Omega,
        \label{eq:navier.stokes:eps:mom} \\ 
        -\imag\omega p^\eps + \rho_0 c^2\Div \vv^\eps &= 0,%
        \hspace{4em}\text{in }\Omega, \label{eq:navier.stokes:eps:pr}\\
        \label{eq:navier.stokes:eps:bound}
        \vv^\eps &= \zerobf, \hspace{4em} \text{on } \partial\Omega.
        \end{align}
    \end{subequations}
    %
    
    Earlier we have proved stability for such a problem for the non-resonant
    case, which we consider here as well, \ie, for vanishing
    viscosity and so absorption, the kernel of the system is empty -- there
    is no eigensolution. %
    The eigenvalues of the limit problem coincide with the Neumann
    eigenvalue of $-\Delta$.
    \begin{lemma}[Stability for the non-resonant case]
        \label{lem:stability:2}
        For any $\vf \in (H^0(\Div,\Omega))\cap H(\scurl,\Omega))'$ the
        system~\eqref{eq:navier.stokes:eps} has a unique solution
        $(\vv^\eps,p^\eps) \in H_0(\Div,\Omega)\cap H(\scurl,\Omega) \times L^2(\Omega)$. %
        If $\tfrac{\omega^2}{c^2}$ is not a Neumann eigenvalue of $-\Delta$, 
        then there exists a constant $C > 0$ independent of $\eps$, such that
        \begin{subequations}
            \label{eq:navier.stokes:stability:nondim}    
            \begin{align}
            \label{eq:navier.stokes:stability:1}    
            \|\vv^\eps\|_{H(\Div,\Omega)} +\eps\,\|\scurl\vv^\eps\|_{L^2(\Omega)} + \|p^\eps\|_{L^2(\Omega)} %
            &\leq C\, \|\vf\|_{(H_0(\Div,\Omega)\cap H(\scurl,\Omega))'}, \\ %
            \label{eq:navier.stokes:stability:2}    
            \|\nabla p^\eps\|_{L^2(\Omega)} &\leq C\,\|\vf\|_{L^2(\Omega)}.
            \intertext{%
                For any $\omega>0$ and for $C^{1,1}$ boundary $\partial\Omega$ it holds }
            \eps\, \|\vv^\eps\|_{(H^1(\Omega))^2} &\leq C\,
            \|\vf\|_{(H_0(\Div,\Omega)\cap H(\scurl,\Omega))'}.
            \label{eq:navier.stokes:stability:3}      	
            \end{align}
        \end{subequations}
    \end{lemma}
    A proof can be found in~\cite[Lemma 2.2]{Schmidt.Thoens.Joly:2014}.
    Even so in this work $C^{\infty}$ was assumed, the proof of
    \eqref{eq:navier.stokes:stability:1},
    \eqref{eq:navier.stokes:stability:2} does not rely on a higher
    regularity assumption, see~\cite{Marusic:2000}.  %
    

    
    \subsection{Asymptotic expansion}
    \label{sec:ImpedanceBC:AsympExpan}
    
    {With the above introduced small parameter %
        $\eps = \sqrt{2\eta/(\omega\rho_0)}$} %
    and using the curvilinear coordinates $(t,s)$, which we have
    introduced in~\eqref{eq:localCoord}, close to the boundary the
    solution of~\eqref{eq:navier.stokes:eps} inspired by the framework of
    Vishik and Lyusternik~\cite{Vishik.Lyusternik:1960} could be written~as
    \begin{align}
    \label{ansaz}
    \vv \!= %
    \sum_{j=0}^\infty \eps^j\left(\vv^j + \eps \plcurl (\phi^j\chi)\right);
    \qquad %
    p \!= \sum_{j=0}^\infty \eps^j p^j,
    \end{align}
    where $\vv^j(x,y)$ and $p^j(x,y)$ are terms of the {\em far field}
    expansion, the {\em near field} terms $\phi^j(t,\frac{s}{\eps})$
    represent the boundary layer close to the wall, $\plcurl
    =(\partial_y,-\partial_x)^\top$, and $\chi$ is an admissible cut-off
    function.
    
    \begin{definition}[Admissible cut-off function]
        We denote a monotone function $\chi \in C^\infty(\Omega)$ an
        admissible cut-off function, %
        if there exist constants $0 < s_1 < s_0 <
        \frac12\|\kappa\|_{L^\infty(\Gamma)}^{-1}$ %
        such that $\chi \equiv 0$ outside an $s_0$-neighbourhood of
        $\partial\Omega$ and otherwise $\chi(\vx) = \wh{\chi}(s)$, where
        $\wh{\chi}(s) = 1$ for $s < s_1$. %
        For an admissible cut-off function $\chi$ we denote $\overline{\supp(\chi)}$ the
        $\chi$-neighbourhood of~$\partial\Omega$.
        \label{lem:chi}
    \end{definition}
    
    The method of multiscale expansion separates the far and near field
    terms. %
    We restrict ourselves to $j=0,1,2$, as these will be used for the derivation of the impedance 
    boundary conditions where the equations for general $j \in \IN$ can be found 
    in~\cite{Schmidt.Thoens.Joly:2014}.
    The far field velocity and pressure terms $(\vv^j,p^j)$ satisfy the PDE system
    \begin{subequations}
        \label{eq:vj}
        \begin{align}
        \nabla \Div \vv^{j} +\frac{\omega^2}{c^2} \vv^j &=  \frac{\imag\omega}{\rho_0 c^2} \vf	\cdot\delta_{j=0} + 
        \frac{\imag\omega^2}{2 c^2}\Delta \vv^{j-2} + \frac{\imag\gamma'\omega^2}{2 c^2}
        \nabla \Div \vv^{j-2}, %
        && \text{ in }\Omega,
        \label{eq:vj:wave}\\
        \vv^j \cdot \vn &= \sum_{\ell=1}^j G_\ell(\partial_t \Div\vv^{j-\ell})
        + H_j(\vf)
        , && \text{ on }\partial\Omega,
        \label{eq:vj:bc}\\
        p^j &= - \frac{\imag\rho_0c^2}{\omega} \Div\vv^j,
        && \text{ in }\Omega,
        \label{eq:vj:p}
        \end{align}
    \end{subequations}
    where $\vv^{-1} = \vv^{-2} = \zerobf$, $G_\ell: C^\infty(\Gamma) \to C^\infty(\Gamma)$ and $H_\ell: C^\infty(\Gamma) \to C^\infty(\Gamma)$ are
    tangential differential operators acting on traces of terms of lower orders or
    the trace of $\vf$ on $\partial\Omega$, respectively. %
    Furthermore, $\delta_{j=0}$ stands for the Kronecker symbol which is
    $1$ if $j=0$ and $0$ otherwise. %
    The operators $G_\ell$ and $H_\ell$ up to $\ell=2$ are given by
    \begin{subequations}
        \begin{align}
        G_0 (v) &= 0, &&
        H_0(\vf) = 0, \\[0.5em]
        %
        G_1(v) &= (1+\imag) \frac{c^2}{2\omega^2}\partial_t v, && 
        %
        H_1(\vf) = -(1-\imag) \frac{1}{2\omega\rho_0}\partial_t(\vf\cdot\vn^\bot), \\
        G_2(v) &= \frac{c^2}{\omega^2}\Big(
        \frac{\imag}{4}
        \partial_t(\kappa v)\Big), && 
        %
        H_2(\vf) = - \frac{1}{4\omega\rho_0}\partial_t(\kappa\,\vf\cdot\vn^\bot).
        \end{align}
    \end{subequations}
    The near field terms %
    $\vv^j_{BL} = \sqrt{\frac{2\,\eta}{\omega\rho_0}}\plcurl(\phi^j \chi)$ for
    $\phi^j(\vx) = \wt{\phi}^j(t,S)$ for $S = s\sqrt{\frac{\omega\rho_0}{2\,\eta}}$
    are defined
    by (\cf~\cite[Lemma~A.1]{Schmidt.Thoens.Joly:2014})
    \begin{align}
    \wt{\phi}^j(t,S) &= \frac{1}{2}(1+\imag)\,\mathrm{e}^{-(1-\imag) S}\sum_{\ell=0}^j (E_\ell(\vv^{j-\ell}\cdot\vn^\bot))(t,S),
    \label{eq:nearField:anyOrder}
    \end{align}
    with the operators $E_\ell: C^\infty(\Gamma) \to C^\infty(\Gamma \times [0,\infty))$ for $\ell = 0,1,2$ 
    \begin{subequations}
        \label{eq:El}
        \begin{align}
        E_0 (v) &= v, \\[0.3em]
        E_1 (v) &= \frac 14 (3+\imag)\,\kappa Sv, \\ 
        %
        E_2 (v) &= \frac{\imag (1+\gamma')\omega^2}{2 c^2}v
        + \frac 14 \big(\imag+(1+\imag)S\big)
        \Big(\frac{3}{4}\kappa^2v +\partial^2_tv \Big)
        + \frac{3}{8}\kappa^2S^2v.
        \end{align}
    \end{subequations}

    \subsection{Derivation of the approximative models and impedance boundary conditions for the velocity} 
    \label{sec:ImpedanceBC:Derivation:v}
    
    Now, we are going to derive the approximative velocity {and
        pressure} models {for $\apprA{\vv}{N}$ and $\apprA{p}{N}$
        including} impedance boundary conditions given in
    Sec.~\ref{sec:ModelDef:velocity}. %
    Let $\vv^{\eps,N}:= \sum_{j=0}^N \eps^n \vv^j$. %
    Then, by~\eqref{eq:vj:bc} we have
    \begin{align*}
    \vv^{\eps,N}\cdot\vn &= %
    \sum_{j=0}^N \eps^j \sum_{\ell=1}^j 
    G_\ell(\partial_t\Div \vv^{j-\ell}) +
    \sum_{j=0}^N \eps^j  H_j(\vf)\\
    &= \sum_{\ell=1}^N \eps^\ell
    \sum_{j=\ell}^{N} \eps^j G_\ell(\partial_t\Div\vv^{j-\ell}) + %
    \sum_{j=0}^N \eps^j H_j(\vf)\\ %
    &= \sum_{\ell=1}^N \eps^\ell
    \sum_{j=0}^{N-\ell} \eps^j G_\ell(\partial_t\Div\vv^j) + %
    \sum_{j=0}^N \eps^j H_j(\vf)\\
    &= \sum_{\ell=1}^N \eps^\ell
    G_\ell(\vv^{\eps,N}) + %
    \sum_{j=0}^N \eps^j H_j(\vf) -
    \eps^{N+1}\sum_{\ell=1}^N \sum_{j=0}^{\ell-1}
    \eps^j G_\ell(\partial_t\Div \vv^{j+1-N-\ell})
    \end{align*}
    Moving all the terms with $\vv^{\eps,N}$ from the right hand side to
    the left hand side and neglecting the terms of order $\eps^{N+1}$ on
    the right hand side and using the equality $\eta = \eps^2\omega\rho_0/2$, %
    we obtain the boundary conditions for $\apprA{\vv}{N}$,
    \begin{align}
    \apprA{\vv}{N}\cdot\vn - 
    \sum_{\ell=1}^N (\sqrt{2\eta/(\omega\rho_0)})^\ell 
    G_\ell(\partial_t\Div\apprA{\vv}{N}) &= 
    \sum_{j=0}^N  (\sqrt{2\eta/(\omega\rho_0)})^j H_j(\vf),
    \label{eq:vappr:N:bc}
    \end{align}
    which is~\eqref{eq:vappr:0:bc}, \eqref{eq:vappr:1:bc} and
    \eqref{eq:vappr:2:bc} for $N=0,1,2$.
    
    To obtain the approximative PDEs we are going to
    simplify~\eqref{eq:vj:wave}. %
    Applying $\scurl$ to~\eqref{eq:vj:wave} we obtain
    \begin{align*}
    \scurl\vv^j &= \frac{\imag}{\omega\rho_0}\scurl\vf\cdot\delta_{j=0}
    - \frac{\imag}{2}\scurl\plcurl\scurl \vv^{j-2}.
    \end{align*}
    By recursion in $j$ we obtain an expression of $\scurl\vv^j$ in
    terms of $\vf$ only (see (2.11) in~\cite{Schmidt.Thoens.Joly:2014}).
    Inserting this expression into~\eqref{eq:vj:wave} we obtain
    \begin{align}
    \nabla \Div \vv^{j} +\frac{\omega^2}{c^2} \vv^j &=
    \sum_{\ell=1}^j L_\ell(\vv^{j-\ell})
    + M_j(\vf)
    \label{eq:vj:wave:simplified}
    \end{align}
    with $L_\ell \equiv 0$ if $\ell\not=2$, $M_j \equiv 0$ if $j$ is odd and otherwise
    \begin{align*}
    L_2 &= \frac{\imag(1+\gamma')\omega^2}{2 c^2} \nabla \Div, &
    M_j &= \frac{\imag \omega}{\rho_0c^2}\left(-\frac{\imag}{2}\plcurl\scurl\right)^{j/2}\vf.
    \end{align*}
    Now, in the same away as above, where $G_\ell$ and $H_j$ are
    replaced by $L_\ell$ and $M_j$, we find~\eqref{eq:vappr:0:PDE},~\eqref{eq:vappr:1:PDE} 
    for the approximative velocities $\apprA{\vv}{N}$, $N=0,1$ and for
    $N \geq 2$
    \begin{align}
    \label{eq:vappr:geq2:PDE}
    \left(1 -\frac{\imag\omega(\eta+\eta')}{\rho_0c^2}\right)
    \nabla\Div \apprA{\vv}{N} + \frac{\omega^2}{c^2}\apprA{\vv}{N}
    &= \sum_{j=0}^N (\sqrt{2\eta/(\omega\rho_0)})^j M_j(\vf),
    \end{align}
    which is equivalent to~\eqref{eq:vappr:2:PDE} for $N = 2$. %
    
    Note, that it is possible to keep a term with
    $\plcurl\scurl\apprA{\vv}{N}$ on the left hand side and with the
    gain of the simple source term $\frac{\imag\omega}{\rho_0c^2}\vf$ on
    the right hand side (for any $N$). %
    However, this PDE needs a further boundary condition, \eg, a
    prescribed trace of $\scurl\apprA{\vv}{N}$ in terms of $\vf$.
    Finally, using~\eqref{eq:navier.stokes:pr} we find the pressure approximation $\apprA{p}{N}$ defined by~\eqref{eq:vN1:pN}
    and in a similar way as the equations for the far field velocity we obtain using~\eqref{eq:nearField:anyOrder} the near field velocity approximation $\apprA{\vw^{BL}}{N}$ defined by~\eqref{eq:vappr:BL}. %

    \subsection{Derivation of the approximative models and impedance boundary conditions for the pressure} 
    \label{sec:ImpedanceBC:Derivation:p}
    
    Taking the divergence of~\eqref{eq:vappr:0:PDE},~\eqref{eq:vappr:1:PDE} for $N=0,1$ we find that $\apprA{p}{N} = -\frac{\imag \rho_0 c^2}{\omega}\,\Div \apprA{\vv}{N}$ 
    satisfies~\eqref{eq:pappr:0:PDE},~\eqref{eq:pappr:1:PDE}. %
    Taking in the same way the divergence of~\eqref{eq:vappr:geq2:PDE} for $N\geq2$ we find that $\apprA{p}{N}$ fulfills for $N \geq 2$
    \begin{align*}
    \left(1 -\frac{\imag\omega(\eta+\eta')}{\rho_0c^2}\right)
    \Delta\apprA{p}{N} + \frac{\omega^2}{c^2}\apprA{p}{N}
    &= \sum_{j=0}^N -\frac{\imag \rho_0 c^2}{\omega} (\sqrt{2\eta/(\omega\rho_0)})^j \Div M_j(\vf) = \Div\vf\ ,
    \end{align*}
    where we used that the divergence of $\plcurl$ vanishes for smooth enough functions. 
    
    Now, for $N = 0$ using~\eqref{eq:vappr:0} we have
    \begin{align*}
    \nabla \apprA{p}{0}\cdot\vn &= -\frac{\imag \rho_0 c^2}{\omega}\,\nabla\Div \apprA{\vv}{0}\cdot\vn
    = \imag \rho_0\omega \apprA{\vv}{0}\cdot\vn + \vf\cdot\vn = \vf\cdot\vn
    \end{align*}
    which is~\eqref{eq:pappr:0:bc}. Similarly, we find for $N = 1$ using~\eqref{eq:vappr:1}
    \begin{align*}
    \nabla \apprA{p}{1}\cdot\vn &= \imag \rho_0\omega \apprA{\vv}{1}\cdot\vn + \vf\cdot\vn = 
    \frac{\imag \rho_0 c^2}{\omega} (1+\imag)\sqrt{\frac{\eta}{2\omega\rho_0}}\partial_t^2\Div\apprA{\vv}{1}\cdot\vn
    + \vf\cdot\vn - (1+\imag)\sqrt{\frac{\eta}{2\omega\rho_0}}\partial(\vf\cdot\vn^\bot)\\
    &= - (1+\imag)\sqrt{\frac{\eta}{2\omega\rho_0}}\partial_t^2\Div\apprA{p}{1}\cdot\vn
    + \vf\cdot\vn - (1+\imag)\sqrt{\frac{\eta}{2\omega\rho_0}}\partial(\vf\cdot\vn^\bot)\ ,
    \end{align*}
    which is~\eqref{eq:pappr:1:bc}. Finally, for $N \geq2$ we obtain using~\eqref{eq:vappr:N:bc},\eqref{eq:vappr:geq2:PDE}
    \begin{align*}
    \left(1 -\frac{\imag\omega(\eta+\eta')}{\rho_0c^2}\right)\nabla \apprA{p}{N}\cdot\vn 
    &= - \left(1 -\frac{\imag\omega(\eta+\eta')}{\rho_0c^2}\right)\frac{\imag \rho_0 c^2}{\omega}\,\nabla\Div \apprA{\vv}{N}\cdot\vn
    = \imag \rho_0\omega \apprA{\vv}{N}\cdot\vn - \frac{\imag \rho_0 c^2}{\omega}\sum_{j=0}^N (\sqrt{2\eta/(\omega\rho_0)})^j M_j(\vf) \\
    &= \imag \rho_0\omega \sum_{\ell=1}^N (\sqrt{2\eta/(\omega\rho_0)})^\ell 
    G_\ell(\partial_t\Div\apprA{\vv}{N}) 
    - \frac{\imag \rho_0 c^2}{\omega}\sum_{j=0}^N (\sqrt{2\eta/(\omega\rho_0)})^j \Big( M_j(\vf) - \frac{\omega^2}{c^2} H_j(\vf) \Big)\\
    &= -\frac{\omega^2}{c^2}\sum_{\ell=1}^N (\sqrt{2\eta/(\omega\rho_0)})^\ell 
    G_\ell(\partial_t\apprA{p}{N}) 
    - \frac{\imag \rho_0 c^2}{\omega}\sum_{j=0}^N (\sqrt{2\eta/(\omega\rho_0)})^j \Big( M_j(\vf) - \frac{\omega^2}{c^2} H_j(\vf) \Big)\ ,
    \end{align*}
    which is~\eqref{eq:pappr:2:bc} for $N = 2$.
    
    In view of~\eqref{eq:vappr:0} and~\eqref{eq:vappr:1} for $N=0,1$ we find that far field velocity approximation $\apprA{\vv}{N}$ 
    can be computed from the pressure approximation $\apprA{p}{N}$ by~\eqref{eq:pj:v:01}. For $N \geq 2$ we obtain a similar relation using~\eqref{eq:vappr:geq2:PDE}
    as
    \begin{align}
    \nonumber
    \apprA{\vv}{N} %
    &= -\frac{c^2}{\omega^2} \left(1 -\frac{\imag\omega(\eta+\eta')}{\rho_0c^2}\right) \nabla \Div \apprA{\vv}{N}
    + \frac{c^2}{\omega^2} \sum_{j=0}^N (\sqrt{2\eta/(\omega\rho_0)})^j M_j(\vf)\\
    &= -\frac{\imag}{\rho_0\omega} \left(1 -\frac{\imag\omega(\eta+\eta')}{\rho_0c^2}\right) \nabla \apprA{p}{N}
    + \frac{c^2}{\omega^2} \sum_{j=0}^N (\sqrt{2\eta/(\omega\rho_0)})^j M_j(\vf)\ ,
    \label{eq:pj:v:N}
    \end{align}
    which is~\eqref{eq:pj:v:2} for $N = 2$.

    %
    %

    
    \section {Justification of the approximative models}
    \label{sec:justification}
    
    In this section we first define canonical approximative pressure and velocity systems that generalizes the derived approximative models and show their well-posedness.
    Even so we derived the approximative models for smooth domains the analysis of the canonical approximative systems requires less regularity. 
    Considering the canonical systems we will not only benefit from a more compact notation, but more general source terms will allow us to prove the bounds on the modelling error. %
    For this higher regularity estimates of the canonical systems are needed.
    
    \subsection{Well-posedness for a canonical approximative pressure system}
    
    In this section we analyse the well-posedness of a class of canonical approximative pressure problems
    \begin{subequations}
        \label{eq:aux_p}
        \begin{align}
        \label{eq:aux_p:PDE}
        \Div \big( \alpha_\eta \nabla p_\eta\big) + 
        \tfrac{\omega^2}{c^2} p_\eta &= -\Div\vgg_\eta, \qquad \text{ in }\Omega, \\[0.2em]
        \label{eq:aux_p:BC}
        \alpha_\eta \nabla p_\eta\cdot\vn - \partial_t \big(\beta_\eta\partial_t p_\eta\big)%
        &= \vgg_\eta\cdot\vn + 
        \partial_t h_\eta,
        \quad \text{ on }\partial\Omega, 
        \end{align}
    \end{subequations}
    with $\alpha_\eta \in \IC$, $\beta_\eta \in L^\infty(\partial\Omega)$. %
    For $\Omega$ smooth enough the weak formulation of~\eqref{eq:aux_p} is given as: Seek $p_\eta \in H^1_{\beta_\eta} := \{ q \in H^1(\Omega) : \sqrt{\beta_\eta}q \in H^1(\partial\Omega)\}$ such that for all $q' \in H^1(\Omega) \cap H^1(\partial\Omega)$
    \begin{align}
    \label{eq:aux_p:var}
    \int_{\Omega} \alpha_\eta \nabla p_\eta \cdot \nabla q' - \tfrac{\omega^2}{c^2} p_\eta q' \,\dx + \int_{\partial\Omega} \beta_\eta\partial_t p_\eta \partial_t q'\,\dsurf
    &= \int_{\Omega} \vgg_\eta \cdot\nabla q'\, \dx + \int_{\partial\Omega} h_\eta \partial_t q'\,\dsurf\ .
    \end{align}
    The approximative pressure systems~\eqref{eq:pappr:1}
    and~\eqref{eq:pappr:2} of order 1 or 2, respectively, belongs to this canonical approximative pressure system. %
    If we indicate the respective functions for the system of order $N$ with a superscript we find that
    \begin{align*}
    \alpha_\eta^1 &= 1, \; %
    \alpha_\eta^2 =  1 - \frac{\imag\omega(\eta + \eta')}{\rho_0c^2}, &
    %
    \beta_\eta^1 &= (1 + \imag)\sqrt{\frac{\eta}{2\omega\rho_0}}, &
    \beta_\eta^2 &= \beta_\eta^1 + \frac{\imag\eta}{2\omega\rho_0}\kappa,\\
    %
    \vgg_\eta^1 &= \vgg_\eta^2 = \vf, &
    %
    h_\eta^1 &= 
    - \frac{1+\imag}{\omega\rho_0}\sqrt{\frac{\eta}{2\omega\rho_0}}\vf\cdot\vn^\bot, &
    h_\eta^2 &= h_\eta^1 + \frac{\imag\omega(\eta + \eta')}{\rho_0c^2}\vf\cdot\vn - \frac{\imag\eta}{2\omega\rho_0}\kappa \vf\cdot\vn^\bot.
    \end{align*}

    \begin{lemma}[Well-posedness of the canonical approximative pressure system]
        \label{lem:aux_p}
        Let $\Omega$ be a Lipschitz domain and $\frac{\omega^2}{c^2}$ be distinct from the Neumann eigenvalues
        of $-\Delta$ in $\Omega$. %
        Moreover, let 
        $\alpha_\eta \in \IC$, $\beta_\eta \in L^\infty(\partial\Omega)$
        $\vgg_\eta\in (L^2(\Omega))^2$, $\beta_\eta^{-\nicefrac12}h_\eta\in L^2(\partial\Omega)$,
        depend continuously on $\eta > 0$, where
        $\Im\ \alpha_\eta \leq 0$, 
        $\alpha_\eta \to 1$ in $L^\infty(\Omega)$,
        $\Im\ \beta_\eta \leq -c|\beta|$ for some $c > 0$. %
        Then, there exists a constant $\eta_m > 0$ such that for any $\eta \in (0,\eta_m)$  the formulation~\eqref{eq:aux_p:var}
        has a unique solution $p_\eta \in H^1(\Omega)$. 
        Furthermore, there exists a constant $C=C(\eta_m) > 0$ not depending on~$\eta$ such that
        \begin{align}
        \label{eq:aux_p:stability}
        \|p_\eta\|_{H^1(\Omega)} 
        &\leq C\, \left(\|\vgg_\eta\|_{(L^2(\Omega))^2} + \|\beta_\eta^{-\nicefrac12} h_\eta\|_{L^2(\partial\Omega)} \right)\ .
        \end{align}
    \end{lemma}
    
    \begin{proof}

        The proof is by contradiction and we suppose, contrary our claim, that the estimate~\eqref{eq:aux_p:stability} is false.
        Then, there exists a sequence $\{\eta_n\}_{n\in\IN}$ with $\eta_n \to 0$, 
        a bounded sequence $\{p_n\}_{n\in\IN}$ with $\|p_n\|_{H^1(\Omega)} = 1$
        and a sequence $\{(\vgg_n, h_n)\}_{n\in\IN}$ with $\|\vgg_n\|_{(L^2(\Omega))^2} + \|\beta_\eta^{-\nicefrac12} h_n\|_{L^2(\partial\Omega)} \to 0$ such that $p_n$ is solution of~\eqref{eq:aux_p:var} 
        where $\vgg_\eta$, $h_\eta$ and $\eta$ are replaced by $g_n$, $h_n$ and $\eta_n$. %
        
        Testing the variational formulation for $p_n$ with $q' = \overline{p_n}$ and taking the imaginary part, we find with the assumptions on $\alpha_\eta$ and $\beta_\eta$ 
        and the Cauchy-Schwarz inequality we find
        \begin{align}
        \label{eq:aux_p:stability:1}
        |\sqrt{\beta_{\eta_n}} p_n|_{H^1(\partial\Omega)}^2 \leq \|\vgg_n\|_{(L^2(\Omega))^2} |p_n|_{H^1(\Omega)} + \|\beta_\eta^{-\nicefrac12} h_n\|_{L^2(\partial\Omega)} |\sqrt{\beta_n} p_n|_{H^1(\partial\Omega)}
        \to 0 \quad \text{ for } n \to \infty\ .
        \end{align}
        
        Then, there exists a weakly converging subsequence, again called $\{p_n\}_{n\in\IN}$, whose limit $p$ for $n\to\infty$
        is the solution of the limit problem:
        \begin{align*}
        \int_\Omega \nabla p\cdot\nabla q' - \tfrac{\omega^2}{c^2} pq'\ \dx &= 0 \quad\forall q' \in H^1(\Omega)\ .
        \end{align*}
        By the assumption on $\omega$ it has as unique solution $p = 0$. Hence,
        \begin{align*}
        p_n \rightharpoonup 0 \quad \text{ in } H^1(\Omega)\ .
        \end{align*}    
        As $H^1(\Omega)$ is compactly embedded in $H^{\nicefrac{1}{2}}(\Omega)$ 
        we have the strong convergence
        \begin{align*}
        p_n \to 0 \quad \text{ in } H^{\nicefrac{1}{2}}(\Omega)
        \end{align*}
        and by the trace theorem
        \begin{align*}
        p_n \to 0 \quad \text{ in } L^2(\partial\Omega)\ .
        \end{align*}
        Finally, testing the variational formulation for $p_n$ with $q' = \overline{p_n}$ we find 
        \begin{align*}
        |p_n|_{H^1(\Omega)}^2 &\leq C \left(\|g_n\|_{L^2(\Omega)} \|p_n\|_{L^2(\Omega)} + 
        \|\beta_\eta^{-\nicefrac12}h_n\|_{L^2(\partial\Omega)} |\sqrt{\beta_\eta}p_n|_{H^1(\partial\Omega)} +  
        \tfrac{\omega^2}{c^2} \|p_n\|_{L^2(\Omega)}^2 + 
        |\sqrt{\beta_{\eta_n}} p_n|_{H^1(\partial\Omega)}^2\right) \to 0\ \text{ for } n \to \infty\ .
        \end{align*}
        This contradicts the assumption and, hence, we have uniqueness and with the Fredholm alternative existence. This completes the proof.
    \end{proof}

    \subsection{Well-posedness for a canonical approximative velocity system}
    
    In this section we analyse the well-posedness of a class of approximative velocity problems
    \begin{subequations}
        \label{eq:aux_v}
        \begin{align}
        \label{eq:aux_v:PDE}
        \nabla \big( \alpha_\eta \Div \vw_\eta\big) + 
        \tfrac{\omega^2}{c^2} \vw_\eta &= \vgg_\eta, \qquad \text{ in }\Omega, \\[0.2em]
        \label{eq:aux_v:BC}
        \vw_\eta\cdot\vn - \partial_t \big(\beta_\eta\partial_t \Div\vw_\eta\big)%
        &= \partial_t h_\eta,
        \quad \text{ on }\partial\Omega, 
        \end{align}
    \end{subequations}
    with $\alpha_\eta \in \IC \backslash \{0\}$, $\beta_\eta \in L^\infty(\partial\Omega)$,
    to which the approximative velocity systems~\eqref{eq:vappr:1}
    and~\eqref{eq:vappr:2} of order 1 or 2, respectively, belongs to. %
    If we indicate the respective functions for the system of order $N$ with a superscript we find that
    \begin{align*}
    \alpha_\eta^1 &= 1, & \alpha_\eta^2 &=  1 - \frac{\imag\omega(\eta + \eta')}{\rho_0c^2}, &
    %
    \beta_\eta^1 &= (1 + \imag)\frac{c^2}{\omega^2}\sqrt{\frac{\eta}{2\omega\rho_0}} &,
    \beta_\eta^2 &=  \beta_\eta^1 + \frac{c^2}{\omega^2}\frac{\imag\eta}{2\omega\rho_0}\kappa,\\
    %
    \vgg_\eta^1 &= \frac{\imag\omega}{\rho_0c^2}\vf, &
    \vgg_\eta^2 &= \vgg_\eta^1 + \frac{\eta}{\rho_0^2c^2}\plcurl\scurl\vf, &
    %
    h_\eta^1 &= \frac{\imag-1}{\omega\rho_0}\sqrt{\frac{\eta}{2\omega\rho_0}}\vf\cdot\vn^\bot, &
    h_\eta^2 &= \alpha^2_\eta\left(h_\eta^1 - \frac{\eta}{2\omega^2\rho_0^2}\kappa\vf\cdot\vn^\bot\right).
    \end{align*}
    
    With $\lambda_\eta = \alpha_\eta \Div\vw_\eta$ on $\partial\Omega$ the variational formulation for~\eqref{eq:aux_v} is given by: Seek $(\vw_\eta, \lambda_\eta) \in H(\Div,\Omega)\times H^1(\partial\Omega)$ such that
    \begin{subequations}
        \label{eq:aux_v:var}
        \begin{align}
        \label{eq:aux_v:var:1}
        \int_\Omega \alpha_\eta \Div\vw_\eta \Div \vv' - \frac{\omega^2}{c^2}\vw_\eta\cdot\vv'\dx - \int_{\partial\Omega} \lambda_\eta \vv'\cdot\vn\dS 
        &= -\int_\Omega \vgg_\eta\cdot\vv'\dx \qquad \forall \vv' \in H(\Div,\Omega)\\
        \label{eq:aux_v:var:2}
        \int_{\partial\Omega} \vw_\eta\cdot\vn \lambda' + \alpha_\eta^{-1} \beta_\eta \partial_t\lambda_\eta \partial_t \lambda' \dS &= -\int_{\partial\Omega} h_\eta\partial_t\lambda'\dS
        \quad \forall \lambda' \in H^1(\partial\Omega).
        \end{align}
    \end{subequations}
    The system~\eqref{eq:aux_v:var} is a saddle point problem with penalty term~\cite[Chap.~III, \S~4]{Braess:2007}. %
    Note, that we can consider~\eqref{eq:aux_v:var} in difference to~\eqref{eq:aux_v} with sources $h_\eta \in L^2(\partial\Omega)$, 
    
    \begin{lemma}[Well-posedness of the canonical approximative velocity system]
        \label{lem:aux_v}
        Let the assumption of Lemma~\ref{lem:aux_p} be fulfilled.
        Then, there exists a constant $\eta_m > 0$ such that for any $\eta \in (0,\eta_m)$  the system~\eqref{eq:aux_v:var}
        has a unique solution $\vw_\eta \in H(\Div,\Omega)$. 
        Furthermore, there exists a constant $C=C(\eta_m) > 0$ not depending on~$\eta$ such that
        \begin{align}
        \label{eq:aux_v:stability}
        \|\vw_\eta\|_{H(\Div,\Omega)} &\leq C\, \left(\|\vgg_\eta\|_{(L^2(\Omega))^2} 
        + \|\beta_\eta^{-\nicefrac12}h_\eta\|_{L^2(\partial\Omega)}
        \right)\ , &
        \|\scurl\vw_\eta\|_{L^2(\Omega)} &\leq C\, \|\scurl\vgg_\eta\|_{L^2(\Omega)}\ .
        \end{align}
    \end{lemma}
    \begin{proof}  
        We start with the Helmholtz decomposition $\vw_\eta = \vw_{\eta,0} + \nabla\psi_\eta$ with $\vw_{\eta,0} \in H_0(\Div0,\Omega)$, $\psi_\eta \in H^1_\star(\Omega) := \{\psi \in H^1(\Omega), \int_\Omega \psi\dx = 0\}$.
        The decomposition is orthogonal since for all $(\vw_0,\psi)\in H_0(\Div0,\Omega) \times H^1(\Omega)$ it holds
        \begin{align*}
        \int_\Omega \nabla \psi \cdot \vw_0\, \dx &= - \int_\Omega \psi \Div\vw_0\,\dx + \int_{\partial\Omega}\psi \vw_0\cdot\vn = 0.
        \end{align*}
        Testing~\eqref{eq:aux_v:var:1} with $\vv' \in H_0(\Div0,\Omega)$ we find that $\vw_{\eta,0}$ is uniquely defined as $c^2/\omega^2$ times the $L^2(\Omega)$-projection of $\vgg_\eta$ onto $H_0(\Div0,\Omega)$.
        Hence, the estimates~\eqref{eq:aux_v:stability} hold for the component $\vw_{\eta,0}$. 
        
        Denoting $\phi_\eta = \Delta\psi_\eta$ where $\lambda_\eta = \alpha_\eta\phi_\eta$ on $\partial\Omega$ we find that it satisfies
        \begin{align*}
        \alpha_\eta\nabla \phi_\eta                 &= \vgg_\eta - \frac{\omega^2}{c^2} \left(\vw_{\eta,0} - \nabla \psi_\eta\right)\\
        \partial_t(\beta_\eta \partial_t \phi_\eta) &= \nabla\psi_\eta\cdot\vn - \partial_t h_\eta \ .
        \end{align*}
        The statement of the lemma is only fulfilled if $\psi_\eta \in H^1_\star(\Omega)$ and, hence, if $\phi_\eta \in H^1(\Omega) \cap H^1(\partial\Omega)$. 
        
        The following of the proof we derive conditions that $\psi_\eta$, $\lambda_\eta$ and $\phi_\eta$ need to fulfill, then define them by variational formulations and show that the defined quantities 
        satisfy the conditions.
        
        Testing~\eqref{eq:aux_v:var:1} with $\vv'=\nabla\psi'$ with $\psi' \in H^1_{\star,\Delta}(\Omega) := \{\phi \in H^1_\star(\Omega) : \Delta\phi \in L^2(\Omega)\}$ and using that $\vw_{\eta,0}\cdot\vn = 0$ on $\partial\Omega$ we find that $(\psi_\eta, \lambda_\eta)$ needs to solve
        \begin{align}
        \label{eq:aux_v:stability:1}
        \int_\Omega \alpha_\eta \Delta \psi_\eta \Delta \psi' - \frac{\omega^2}{c^2} \nabla \psi_\eta \cdot \nabla \psi' \,\dx - \int_{\partial\Omega} \lambda_\eta \nabla\psi'\cdot\vn \,\dsurf
        &= -\int_{\Omega} \vgg_\eta \cdot \nabla \psi'\,\dx \quad \forall \psi' \in H^1_{\star,\Delta}(\Omega)\ .
        \end{align}
        Now, inserting the definition of $\phi_\eta$ integration by parts we find that $\phi_\eta$ satisfies (in case of higher regularity of $h_\eta$)
        \begin{align}
        \nonumber
        \int_\Omega \vgg_\eta\cdot\nabla\psi'\,\dx &= %
        \int_\Omega -\alpha_\eta \phi_\eta\Delta\psi' - \frac{\omega^2}{c^2} \Delta\psi_\eta \psi'\,\dx + \int_{\partial\Omega}\lambda_\eta\nabla\psi'\cdot\vn + \frac{\omega^2}{c^2}\nabla\psi_\eta\cdot\vn\psi'\,\dsurf,\\
        \label{eq:aux_v:stability:2}
        &= \int_\Omega \alpha_\eta \nabla\phi_\eta\cdot\nabla\psi' - \frac{\omega^2}{c^2} \phi_\eta \psi'\,\dx + \frac{\omega^2}{c^2}\int_{\partial\Omega} \nabla\psi_\eta\cdot\vn\psi'\,\dsurf,
        \end{align}
        where we have used $\lambda_\eta = \alpha_\eta\phi_\eta$ on $\partial\Omega$ in the last step. %
        Inserting the decomposition of $\vw_\eta$ into~\eqref{eq:aux_v:var:2} and using~$\lambda_\eta = \alpha_\eta\phi_\eta$ on $\partial\Omega$ we find
        \begin{align}
        %
        \label{eq:aux_v:stability:3}
        \int_{\partial\Omega} \nabla\psi_\eta\cdot\vn \lambda' + \beta_\eta \partial_t\phi_\eta\partial_t\lambda' \,\dsurf
        &= -\int_{\partial\Omega} h_\eta \partial_t \lambda'\,\dsurf\ \quad\forall \lambda' \in H^1(\partial\Omega)
        \end{align}
        Subtracting $\omega^2/c^2$ times~\eqref{eq:aux_v:stability:3} from~\eqref{eq:aux_v:stability:2} for test functions $\phi' \in H^1_{\star,\Delta}(\Omega)\cap H^1(\partial\Omega)$
        we see that $\phi_\eta$ needs to satisfy
        \begin{align}
        \int_\Omega \alpha_\eta \nabla\phi_\eta\cdot\nabla\phi' - \frac{\omega^2}{c^2} \phi_\eta \phi'\,\dx 
        - \frac{\omega^2}{c^2}\int_{\partial\Omega}\beta_\eta \partial_t\phi_\eta\partial_t\phi' \,\dsurf
        &= \int_\Omega \vgg_\eta\cdot\nabla\phi'\,\dx + \frac{\omega^2}{c^2}\int_{\partial\Omega} h_\eta \partial_t \phi'\,\dsurf\ .
        \label{eq:aux_v:stability:4}
        \end{align}
        Considering~\eqref{eq:aux_v:stability:4} as variational formulation in $H^1_{\star}(\Omega) \cap H^1(\partial\Omega)$ and 
        following the lines of the proof of Lemma~\ref{lem:aux_p} we see that this formulation provides a unique solution $\phi_\eta\in H^1_{\star}(\Omega) \cap H^1(\partial\Omega)$
        with
        \begin{align*}
        \|\phi_\eta\|_{H^1(\Omega)} 
        \leq C \left(\|\vgg_\eta\|_{(L^2(\Omega))^2} + \|\beta_\eta^{-\nicefrac12} h_\eta\|_{L^2(\partial\Omega)}\right)\ .
        \end{align*}
        Inserting $\phi_\eta = \Delta\psi_\eta$ into~\eqref{eq:aux_v:stability:1} and using integration by parts 
        together with $\lambda_\eta = \alpha_\eta\phi_\eta$ on $\partial\Omega$ 
        we find that we $\psi_\eta$ can be defined uniquely by the variational formulation: Seek $\psi_\eta \in H^1_\star(\Omega)$ such that
        \begin{align}
        \int_\Omega \nabla\psi_\eta \cdot\nabla\psi'\,\dx &= \frac{c^2}{\omega^2}\int_\Omega \left(\vgg_\eta - \alpha_\eta \nabla\phi_\eta\right)\cdot\nabla\psi'\,\dx
        \quad\forall \psi' \in H^1_\star(\Omega)\ 
        \end{align}
        with
        \begin{align*}
        \|\psi_\eta\|_{H^1(\Omega)} \leq C \left(\|\vgg_\eta\|_{(L^2(\Omega))^2} + \|\beta_\eta^{-\nicefrac12} h_\eta\|_{L^2(\partial\Omega)}\right)\ .
        \end{align*}
        As $H^1_{\star,\Delta}(\Omega) \subset H^1_{\star}(\Omega)$ it follows that $\psi_\eta$ fulfills~\eqref{eq:aux_v:stability:1} as well and, hence,~\eqref{eq:aux_v:stability:3}.
        
        Hence, $\vw_\eta = \vw_{\eta,0} + \nabla\psi_\eta$ fulfills~\eqref{eq:aux_v:var} and the first estimate of the lemma. %
        Finally, taking test functions $\vv' = \plcurl\psi'$ with $\psi'\in H^1(\Omega)$ we find that $\scurl\vw_\eta = \frac{c^2}{\omega^2}\scurl\vgg_\eta$ and so the second estimate.
        This completes the proof.

    \end{proof}

    \subsection{Well-posedness and equivalence of approximative models for pressure and velocity}
    
    With the well-posedness of the canonical approximative models for pressure and velocity we are position to prove the well-posednes as well as the equivalence of the 
    approximative models.
    
    \begin{proof}[Proof of Theorem~\ref{lem:stability}] 
        The well-posedness of the approximative models~\eqref{eq:pappr:0} and~\eqref{eq:vappr:0} of order~0 were proven in~\cite{Schmidt.Thoens.Joly:2014}. %
        
        The well-posedness of the approximative models~\eqref{eq:pappr:1},~\eqref{eq:pappr:2} for $\apprA{p}{N}$, $N=1,2$
        and~\eqref{eq:vappr:1},~\eqref{eq:vappr:2} for $\apprA{\vv}{N}$, $N=1,2$ follows from Lemma~\ref{lem:aux_p} and Lemma~\ref{lem:aux_v},
        where the assumption on the smoothness of the boundary $\partial\Omega$ guarantees that the curvature $\kappa \in L^\infty(\partial\Omega)$.
        
        It remains to show that $\apprA{p}{N}$ defined by~\eqref{eq:vN1:pN} and with~\eqref{eq:pappr:0}--\eqref{eq:pappr:2} are equivalent
        as well as $\apprA{\vv}{N}$ defined by~\eqref{eq:pN1:vN} and with~\eqref{eq:vappr:0}--\eqref{eq:vappr:2}.
        
        With $\vf\in L^2(\Omega)$ it follows that $\Div\apprA{\vv}{N} \in H^1(\Omega)$ and $\apprA{p}{N} := -\frac{\imag\rho_0c^2}{\omega}\Div\apprA{\vv}{N} \in H^1(\Omega)$ fulfills~\eqref{eq:pappr:0}--\eqref{eq:pappr:2}
        -- due to the derivation of~\eqref{eq:pappr:0}--\eqref{eq:pappr:2} in Sec.~\ref{sec:ImpedanceBC:Derivation:p}. Hence, $\apprA{p}{N}$ defined by~\eqref{eq:vN1:pN} and with~\eqref{eq:pappr:0}--\eqref{eq:pappr:2} are equivalent.%
        
        Moreover, applying the divergence to~\eqref{eq:pN1:vN} and inserting~\eqref{eq:pappr:0}--\eqref{eq:pappr:2} we find
        \begin{align*}
        \Div \apprA{\vv}{N} %
        &= \frac{\imag}{\rho_0\omega}\Div\vf - \frac{\imag}{\rho_0\omega} \Big(1- \frac{\imag\omega(\eta+\eta')\delta_{N=2}}{\rho_0 c^2}\Big) \Delta \apprA{p}{2} 
        = \frac{\imag\omega}{\rho_0 c^2} \apprA{p}{N} \in H^1(\Omega).
        \end{align*}
        Then, applying $\nabla$, multiplying with $(1- \frac{\imag\omega(\eta+\eta')\delta_{N=2}}{\rho_0 c^2})$ and using~\eqref{eq:pN1:vN} we obtain
        \begin{align*}
        \Big(1- \frac{\imag\omega(\eta+\eta')\delta_{N=2}}{\rho_0 c^2}\Big) \nabla \Div \apprA{\vv}{N} %
        &= \frac{\imag\omega}{\rho_0 c^2} \Big(1- \frac{\imag\omega(\eta+\eta')\delta_{N=2}}{\rho_0 c^2}\Big) \nabla \apprA{p}{N}\\
        &= \frac{\omega^2}{c^2} \left(-\apprA{\vv}{N} + \frac{\imag}{\rho_0\omega}\vf + \frac{\eta}{\rho_0^2\omega^2}\delta_{N=2}\plcurl\scurl\vf\right)
        \end{align*}
        which is~\eqref{eq:vappr:0:PDE}, \eqref{eq:vappr:1:PDE}, or~\eqref{eq:vappr:2:PDE}, respectively.
        
        Taking the normal trace of~\eqref{eq:pN1:vN} on $\partial\Omega$ and inserting~\eqref{eq:pappr:0:bc}, \eqref{eq:pappr:1:bc}, or~\eqref{eq:pappr:2:bc}, respectively, we find 
        \begin{align*}
        \apprA{\vv}{N}\cdot\vn &= 
        \frac{\imag}{\rho_0\omega} \left(\vf\cdot\vn - \Big(1- \frac{\imag\omega(\eta+\eta')\delta_{N=2}}{\rho_0 c^2}\Big) \nabla \apprA{p}{N} \cdot\vn\right)
        + \frac{\eta}{\rho_0^2\omega^2}\delta_{N=2}\plcurl\scurl\vf\cdot\vn\\
        &= \frac{\imag}{\rho_0\omega} \left((1+\imag)\sqrt{\frac{\eta}{2\omega\rho_0}}
        \delta_{N\geq1} \left(\partial_t^2 \apprA{p}{N} + \partial_t(\vf\cdot\vn^\bot)\right)
        + \frac{\imag\eta}{2\omega\rho_0} \delta_{N=2} \left(\partial_t(\kappa\partial_t \apprA{p}{N}) + \partial_t(\kappa \vf\cdot\vn^\bot)\right) 
        \right)\in H^{-\nicefrac{1}{2}}(\partial\Omega)\ ,
        \end{align*}
        since with~\eqref{eq:pappr:1:bc} we have $\partial_t^2 \apprA{p}{1} \in H^{-\nicefrac{1}{2}}(\partial\Omega)$
        and with~\eqref{eq:pappr:2:bc} it follows $\partial_t^2 \apprA{p}{2} + \frac{1 + \imag}{2} \sqrt{\frac{\eta}{2\omega\rho_0}}\partial_t(\kappa\partial_t \apprA{p}{N}) \in H^{-\nicefrac{1}{2}}(\partial\Omega)$. %
        Now, taking the trace of~\eqref{eq:pN1:vN} on $\partial\Omega$ and inserting in the previous identity we find
        \begin{align*}
        \apprA{\vv}{N}\cdot\vn &= 
        (1+\imag)\sqrt{\frac{\eta}{2\omega\rho_0}}\delta_{N\geq1}
        \left( \frac{\omega^2}{c^2} \partial_t^2 \Div\apprA{\vv}{N} 
        + \frac{\imag}{\rho_0\omega}\partial_t(\vf\cdot\vn^\bot)\right)
        + \frac{\imag\eta}{2\omega\rho_0}\delta_{N=2}
        \left( \frac{\omega^2}{c^2}\partial_t(\kappa\partial_t \Div\apprA{\vv}{N})
        + \frac{\imag}{\rho_0\omega}\partial_t(\kappa \vf\cdot\vn^\bot)
        \right)\ ,
        \end{align*}
        which is~\eqref{eq:vappr:0:bc}, \eqref{eq:vappr:1:bc}, or~\eqref{eq:vappr:2:bc}, respectively. 
        Hence, $\apprA{\vv}{N}$ defined by~\eqref{eq:pN1:vN} and with~\eqref{eq:vappr:0}--\eqref{eq:vappr:2} are equivalent.%
        
        This finishes the proof.
    \end{proof}

    \subsection{Modelling error of the approximative models}
    \label{sec:justification:asympExactness}
    
    In this section we show in Lemma~\ref{lem:error:farfield} that the approximative solutions of order $N$ are asympotically close to the asymptotic far field expansions of the exact solution. 
    For this we need some higher regularity of the terms of the asymptotic expansion (Lemma~\ref{lem:vj:regularity}).
    As the asymptotic expansions are justified the estimates for the modelling error follow immediately.
    

    
    \begin{lemma}
        \label{lem:vj:regularity}
        Let the assumptions of Theorem~\ref{lem:error} be fulfilled.
        Then, there exists a neighbourhood $\Omega_\Gamma$ of $\partial\Omega$
        such that
        for any $j \in\IN_0$ and any $m \in \IN_0$ it holds $\Div\vv^j \in
        H^2(\Omega)\cap H^{m}(\Omega_\Gamma)$.
    \end{lemma}
    \begin{proof}
        By Lemma 2.3 in~\cite{Schmidt.Thoens.Joly:2014} all terms $\vv^j
        \in (H^1(\Omega))^2$ and by Lemma~4.6
        in~\cite{Schmidt.Thoens.Joly:2014} the terms $\vv^j$ have any
        Sobolev regularity in any subdomain of $\Omega_\Gamma$ of
        $\wt{\Omega}$. %
        Using~\eqref{eq:vj:wave} and (2.11)
        in~\cite{Schmidt.Thoens.Joly:2014} {we find by induction in $j$}
        \begin{align*}
        \nabla\Div\vv^j &= %
        -\frac{\omega^2}{c^2}\vv^j %
        - \frac{\imag\omega}{\rho_0c^2}\vf\cdot\delta_{j=0} +
        \frac{\imag(1+\gamma')\omega^2}{2c^2}\nabla\Div\vv^{j-2}
        - \frac{\imag\omega^2}{2c^2}\plcurl\scurl\vv^{j-2}\\
        &= -\frac{\omega^2}{c^2}\vv^j - \delta_{j \text{ is
                even}}\frac{\imag\omega}{\rho_0c^2}\left(-\frac{\imag}{2}\plcurl\scurl\right)^{j/2}\vf
        +
        \frac{\imag(1+\gamma')\omega^2}{2 c^2}\nabla\Div\vv^{j-2}
        \in (H^1(\Omega))^2 \cap (H^{m-1}(\Omega_\Gamma))^2,
        \end{align*}
        and so the statement of the lemma.
    \end{proof}
    
    \begin{lemma}
        \label{lem:error:farfield}
        Let the assumptions of Theorem~\ref{lem:error} be fulfilled.
        Then, it holds for the solution $\apprA{\vv}{N}$ of the approximative models~\eqref{eq:vappr:1} and~\eqref{eq:vappr:2} for $N=1,2$, respectively, 
        that 
        $\scurl \apprA{\vv}{N} - \sum_{j=0}^N (\frac{2\eta}{\omega\rho_0})^{\frac{j}{2}}\scurl \vv^j = 0$
        and, 
        there exist constants $\eta_0$ and
        $C$ independent of $\eta$ such that for 
        $\apprA{\vv}{N}$ and for $\apprA{p}{N}$ for $N=1,2$ given by~\eqref{eq:vN1:pN}
        and any
        $\eta \in (0,\eta_0)$ it holds
        \begin{align}
        \Big\| \apprA{\vv}{N} - \sum_{j=0}^N
        \bigg(\frac{2\eta}{\omega\rho_0}\bigg)^{\frac{j}{2}}\vv^j\Big\|_{H(\Div,\Omega)} +
        \Big\|\apprA{p}{N} - \sum_{j=0}^N
        \bigg(\frac{2\eta}{\omega\rho_0}\bigg)^{\frac{j}{2}}p^j\Big\|_{H^1(\Omega)}
        &\leq C \eta^{\frac{N+1}{2}}\ .
        \label{eq:error:farfield}
        \end{align}
    \end{lemma}
    
    \begin{proof}
        Comparing the governing equations for $\apprA{\vv}{N}$ with those for $\vv^j$, \ie,~\eqref{eq:vappr:N:bc} with~\eqref{eq:vj:bc} and~\eqref{eq:vappr:1:PDE},~\eqref{eq:vappr:geq2:PDE} with~\eqref{eq:vj:wave:simplified}
        we claim that $\apprA{\vv}{N}$ has an asymptotic expansion in the form
        \begin{align}
        \apprA{\vv}{N} \approx
        \vv^{N,0}+\sqrt{\tfrac{2\eta}{\omega\rho_0}}\vv^{N,1} +
        \tfrac{2\eta}{\omega\rho_0} \vv^{N,2} +
        \left(\tfrac{2\eta}{\omega\rho_0}\right)^{\nicefrac{3}{2}} \vv^{N,3} + \ldots,
        \label{eq:v:decomposition}
        \end{align}
        where $\vv^{N,j} := \vv^j$ for $j = 0,1,\ldots,N$. To justify this asymptotic expansion we call
        \begin{align}
        \label{eq:delta_v_N_n}
        \delta\apprA{\vv}{N,n} = \apprA{\vv}{N} - \sum_{j=0}^n \left(\tfrac{2\eta}{\omega\rho_0}\right)^{\nicefrac{j}{2}}\vv^{N,j},
        \end{align}
        the remainder of order $N$ and estimate it in norm in powers of $\sqrt{\eta}$.
        
        First, the terms $\vv^{N,j}$, $j \geq N+1$ satisfy
        \begin{align}
        \begin{aligned}
        \nabla\Div \vv^{N,j} + \frac{\omega^2}{c^2} \vv^{N,j} &= 
        \delta_{N\geq2}\frac{\imag(1+\gamma')\omega^2}{2c^2}\nabla\Div \vv^{N,j-2} =: \vf^{N,j}
        , && \text{in } \Omega,\\ 
        \vv^{N,j}\cdot \vn &= \sum_{\ell=1}^N G_\ell(\partial_t\Div\vv^{N,j-\ell}) =: g^{N,j},  && \text{on } \partial\Omega\ .
        \end{aligned}
        \label{eq:v_N_j}
        \end{align}
        For $j = N+1$ there are only terms of $\vv^{N,j} = \vv^j$, $j \leq N$ on the right hand side. For those terms by Lemma~\ref{lem:vj:regularity} 
        and the trace theorem we have $\Div\vv^{N,j} \in H^2(\Omega)$ and $\Div\vv^{N,j} \in H^{m+\nicefrac{1}{2}}(\partial\Omega)$ for any $m \in \IN$
        and any $j \neq N$. %
        Hence, for the right hand sides for $j=N+1$ we have the regularity $\vf^{N,j} \in (L^2(\Omega))^2$ and $g^{N,j} \in H^{m+\nicefrac{1}{2}}(\partial\Omega)$ for any $m \in \IN$. %
        By Lemma~2.3 in \cite{Schmidt.Thoens.Joly:2014} $\vv^{N,j}$, $j=N+1$ is well defined and by Lemma~4.6 in~\cite{Schmidt.Thoens.Joly:2014} 
        it has the regularity $\Div\vv^{N,j} \in H^2(\Omega)$ and $\Div\vv^{N,j} \in H^{m+\nicefrac{1}{2}}(\partial\Omega)$ for any $m \in \IN$.
        As the right hand side for $j=N+2$ consists only of terms $\vv^{N,j}$, $j \leq N+1$ it has the same regularity as $j=N+1$.
        Now, by induction in $j$ all terms $\vv^{N,j}$ are well defined and $\Div\vv^{N,j} \in H^2(\Omega)$ and $\Div\vv^{N,j} \in H^{m+\nicefrac{1}{2}}(\partial\Omega)$ for any $m \in \IN$.
        
        Inserting the decomposition~\eqref{eq:v:decomposition} of $\apprA{\vv}{N}$ in their governing equations~\eqref{eq:vappr:1:PDE},~\eqref{eq:vappr:geq2:PDE} 
        and~\eqref{eq:vappr:N:bc} and using the governing equations for $\vv^j$ and $\vv^{N,j}$ we find that the remainder $\delta\apprA{\vv}{N,n}$ fulfills
        \begin{align}
        \label{eq:delta_vv_N_n:problem}
        \begin{aligned}
        \left(1 - \frac{\imag\omega(\eta+\eta')\delta_{N\geq2}}{\rho_0c^2}\right)\nabla\Div\delta\apprA{\vv}{N,n}
        + \frac{\omega^2}{c^2} \delta\apprA{\vv}{N,n} 
        &= \frac{\imag(1+\gamma')\omega^2}{2c^2}\delta_{N\geq2} 
        \sum_{j=n+1}^{n+2} \left(\frac{2\eta}{\omega\rho_0}\right)^{j/2} \nabla\Div \vv^{N,j-2},
        &&\text{ in }\Omega\\
        \delta\apprA{\vv}{N,n}\cdot\vn - \sum_{\ell=1}^N \left(\frac{2\eta}{\omega\rho_0}\right)^{j/2}  G_\ell(\partial_t\Div \delta\apprA{\vv}{N,n})
        &= - \sum_{j=n+1}^{n+N} \left(\frac{2\eta}{\omega\rho_0}\right)^{j/2} \sum_{\ell=j-n}^{N} G_\ell(\partial_t\Div\vv^{N,j-\ell}),
        &&\text{ on }\partial\Omega\ .
        \end{aligned}
        \end{align}
        The problem~\eqref{eq:delta_vv_N_n:problem} for the remainder for $N=1,2$ belongs to the canonical approximative velocity problem~\eqref{eq:aux_v} and with Lemma~\ref{lem:aux_v} we find
        for all $n \geq N$ that $\scurl\delta\apprA{\vv}{N,n} = 0$ and
        \begin{align}
        \label{eq:delta_vv_N_n:estimate}
        \|\delta\apprA{\vv}{N,n}\|_{H(\Div,\Omega)} &\leq C \eta^{\nicefrac{(2n+1)}{4}}\ .
        \end{align}
        Finally, for $n = N+1$ we obtain 
        \begin{align*}
        \|\delta\apprA{\vv}{N,N}\|_{H(\Div,\Omega)} &\leq \left(\frac{2\eta}{\omega\rho_0}\right)^{\frac{N+1}{2}} \|\vv^{N,N+1}\|_{H(\Div,\Omega)} + \|\delta\apprA{\vv}{N,N+1}\|_{H(\Div,\Omega)}
        \end{align*}
        and with~\eqref{eq:delta_vv_N_n:estimate} the bounds for the velocity follows. %
        Moreover, with the definition~\eqref{eq:vN1:pN} of the pressure approximation 
        and the definition~\eqref{eq:vj:p} of the terms of the asymptotic expansion of the pressure
        the same bound follows for the $L^2(\Omega)$-norm of the pressure.
        Finally, the $H^1(\Omega)$-bound follows from the equations~\eqref{eq:pappr:1} and~\eqref{eq:pappr:2} for the approximative pressure and 
        and respective equations for the terms $p^j$ of the asymptotic pressure expansion that is derived using~\eqref{eq:vj:wave:simplified} and~\eqref{eq:vj:p}. %
        That finishes the proof.
    \end{proof}
    
    Now, we are in the position to prove the 
    estimates on
    the modelling error for the approximative solutions.

    \begin{proof}[Proof of Theorem~\ref{lem:error}]
        The estimates~\eqref{lem:error} 
        follow immediately from Lemma~2.2 of~\cite{Schmidt.Thoens.Joly:2014}, Lemma~\ref{lem:error:farfield},
        and the triangle inequality. %
    \end{proof}

    
    \begin{figure}[tb]
        \centering
        \parbox[t]{2.5cm}{\includegraphics[width=2.5cm]{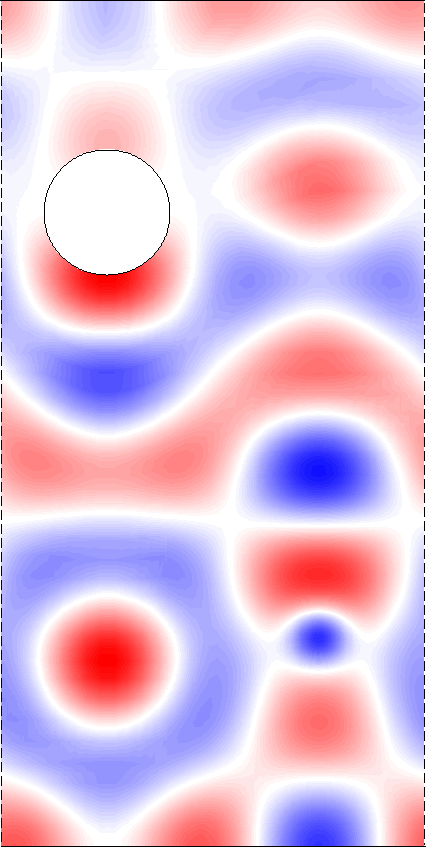} 
            \\\centering{Order $N=0$}}
        \hfill
        \parbox[t]{2.5cm}{\includegraphics[width=2.5cm]{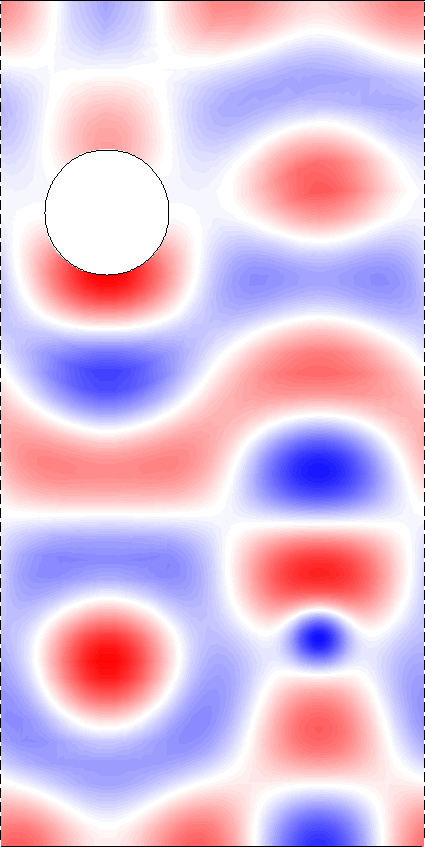} 
            \\\centering{Order $N=1$}}
        \hfill
        \parbox[t]{2.5cm}{\includegraphics[width=2.5cm]{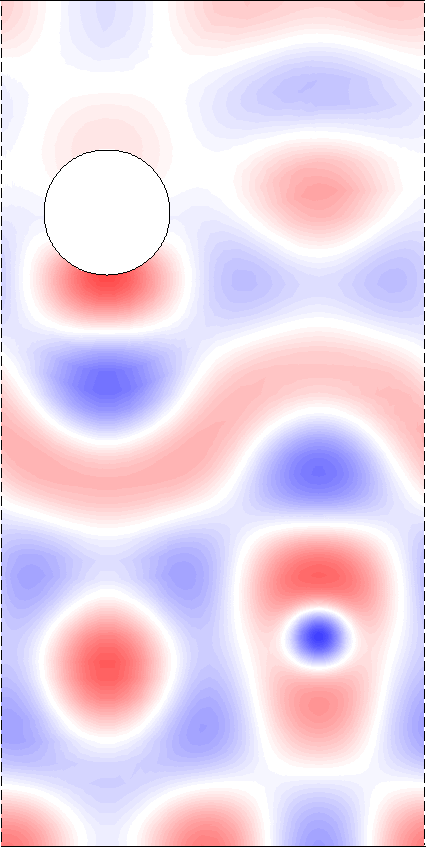} 
            \\\centering{Order $N=2$}} 
        \hfill
        \parbox[t]{2.5cm}{\includegraphics[width=2.5cm]{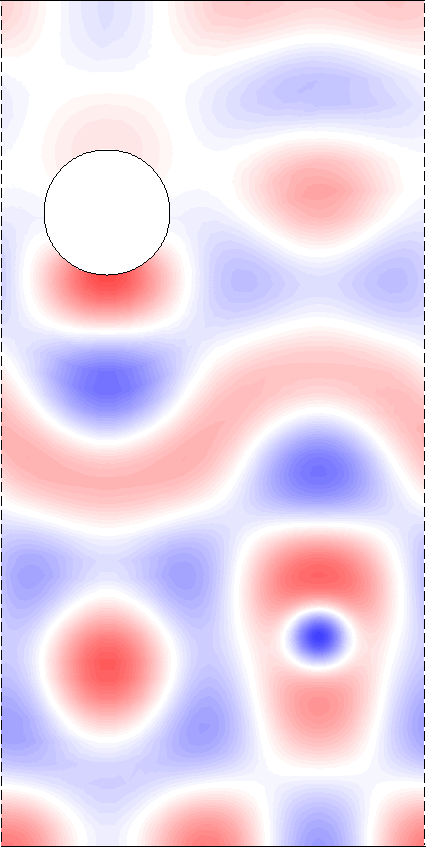} 
            \\\centering{Exact model}} 
        \parbox[b]{1cm}{
            \begin{tikzpicture}
            \begin{axis}[
            enlargelimits=false, 
            scale only axis,
            width=0.3cm,
            height=5cm,
            xticklabels={}
            ]
            \addplot graphics
            [xmin=0,xmax=1,ymin=-1.13,ymax=1.13]
            {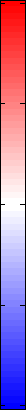};
            \end{axis}
            \end{tikzpicture} 
        } 
        \hfill
        \parbox[t]{2.5cm}{\includegraphics[height=5.0cm,width=2.5cm]{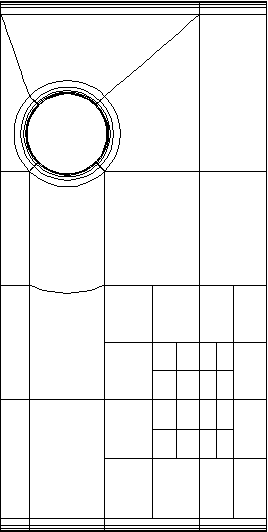} 
            \\\centering{Mesh}   } 
        \caption[Comparison]{Comparison of the real part of the pressure offer
            the approximate models of order $N=0, 1, 2$ to the exact pressure
            ($\eta=4\cdot10^{-6}$, $\omega=15$).  The mesh
            resolving the boundary layers used in the FEM of higher order is
            shown in the right subfigure.}
        \label{fig:pressure}
    \end{figure}
    
    \section{Numerical results}
    \label{sec:NumResults}
    
    For a torus domain with omitted disk, see Fig.~\ref{fig:torus}, we
    have performed numerical simulations for the exact
    model~\eqref{eq:navier.stokes} and the approximative pressure
    models~\eqref{eq:pappr:0}--\eqref{eq:pN1:vN}.  %
    We consider the problem in dimensionless quantities. %
    The domain is the rectangle $[0,1]\times[0,2]$, where the left and
    right sides are identified with each other, and the disk of diameter
    $0.30$ is centered at $(0.25,1.5)$. %
    As source $\vf$ we use the gradient of the Gaussian %
    $\exp(-|\vx - \vx_0|^2/0.005)$ with $\vx_0 = (0.75, 0.5)^\top$. %
    The source is $\scurl$-free, which has no influence to any of the numerical experiments.
    Furthermore, we choose for the speed of sound $c = 1$, %
    the (mean) air density as $\rho_0 = 1$ and neglect the second
    viscosity, $\eta' = 0$. %
    
    For the simulation we have used high-order finite elements within the
    numerical C++ library \emph{Concepts}~\cite{conceptsweb} to push the
    discretisation error below the modelling error. %
    We use $C^0$-continuous finite elements for the (approximative) pressure and both
    components of the (exact) velocity. Note, that the classical choice for the approximative velocity models 
    are $H(\Div,\Omega)$-conforming finite elements like Raviart-Thomas elements. 
    Here, we restrict the numerical experiments to the models of the approximative pressure which provides the greatest simplification.
    
    To resolve the boundary layers in the (exact) velocity, we refine the mesh
    geometrically towards the boundary, see the right picture in
    Fig.~\ref{fig:pressure}. %
    The high gradients of the source term are considered in a further
    (geometric) mesh refinement towards the point $\vx_0$. %
    The far field solution of the approximative models could be computed
    to a high precision on a rather coarse mesh as no boundary layer has
    to be resolved. %
    Anyhow, we have computed the far field solution on the mesh
    illustrated in Fig.~\ref{fig:pressure}, which allowed us firstly a
    straightforward evaluation of norms of the error functions and
    secondly a representation of the sum of far and near field on the same
    mesh. %
    We have chosen the polynomial degree to be 11 to obtain low enough
    discretisation errors such that the modelling errors become visible.
    
    \begin{figure}[tb]  
        {\parbox[t]{0.99\linewidth}{
                \begin{tikzpicture}[scale=0.96]
                \pgfplotsset{ 
                    legend cell align=left,
                    legend style={
                        font=\footnotesize, anchor=north east} %
                } %
                \begin{axis}[%
                xlabel=$x_2$,
                xmin=0,xmax=2,
                ylabel=$\Im\, v_1$,
                legend style={draw=none}, 
                legend style={fill=none},
                line width=1pt,
                extra x ticks={0.15},  
                extra x tick style={    
                    xticklabels={\tiny 0.15}, 
                    xmajorgrids=true            
                }
                ]
                \addplot+[smooth,mark=none,black,thin] table [x = y, y expr=\thisrow{exact}/1.77e-4] {Hole4em2Results.dat};
                \addplot+[smooth,mark=none,red,style=dashed,thin] table [x = y, y expr=\thisrow{FF}/1.77e-4] {Hole4em2Results.dat};
                \addlegendentry{exact solution};
                \addlegendentry{far field approximation};
                \end{axis}
                \end{tikzpicture}
                \hspace{-21em}(a)\hspace{21em}
                \begin{tikzpicture}[scale=0.96]
                \pgfplotsset{ 
                    legend cell align=left,
                    legend style={
                        font=\footnotesize, anchor=north east} %
                } %
                \begin{axis}[%
                xlabel=$x_2$,
                xmin=0,xmax=0.15,
                ylabel=$\Im\, v_1$,
                ymin=-0.13,ymax=0.13,
                xtick={0,0.05,0.1,0.15},
                xticklabels={$0$,$0.05$,$0.1$,$0.15$},
                yticklabels={,$-0.1$,$-0.05$,$0$,$0.05$,$0.1$},
                legend style={draw=none}, 
                line width=1pt
                ]
                \addplot+[smooth,mark=none,black,thin] table [x = y, y expr=\thisrow{exact}/1.77e-4] {Hole4em2Results.dat};
                \addlegendentry{exact solution};
                \addplot+[smooth,mark=none,red,style=dashed,thin] table [x = y, y expr=\thisrow{FF}/1.77e-4] {Hole4em2Results.dat};
                \addlegendentry{far field approximation};
                \addplot+[smooth,mark=none, red,thin,style=dash dotted] table [x = y, y expr=\thisrow{NF}/1.77e-4] {Hole4em2Results.dat};
                \addlegendentry{correcting near field};
                \addplot+[smooth,mark=none, red,thin] table [x = y, y expr=\thisrow{approximate}/1.77e-4] {Hole4em2Results.dat};
                \addlegendentry{far field + near field};
                \addplot+[smooth,mark=none,black,thin] table [x = y, y expr=\thisrow{exact}/1.77e-4] {Hole4em2Results.dat};      
                \addplot+[smooth,mark=none,red,style=dashed,thin] table [x = y, y expr=\thisrow{FF}/1.77e-4] {Hole4em2Results.dat};
                \end{axis}
                \end{tikzpicture}
                \hspace{-22em}(b)\hspace{12em}
        }}
        \caption[Normal velocity]{%
            Imaginary part of first velocity component in side view for
            $x_1=0$ %
            with \mbox{$\sqrt{\eta}=4\cdot10^{-2}$}, which is at $x_2 = 0$
            tangential to the bottom wall. The exact solution $v_1$ and the
            approximate {(far field) solution $(\apprA{v}{2})_1$ of order
                2, the corresponding near field $(\apprA{v^{BL}}{2})_1$ and the
                sum of both are shown, in (a) for the whole line $x_1 = 0$, and
                in (b) close to the wall.}}
        \label{fig:sideview}
    \end{figure}
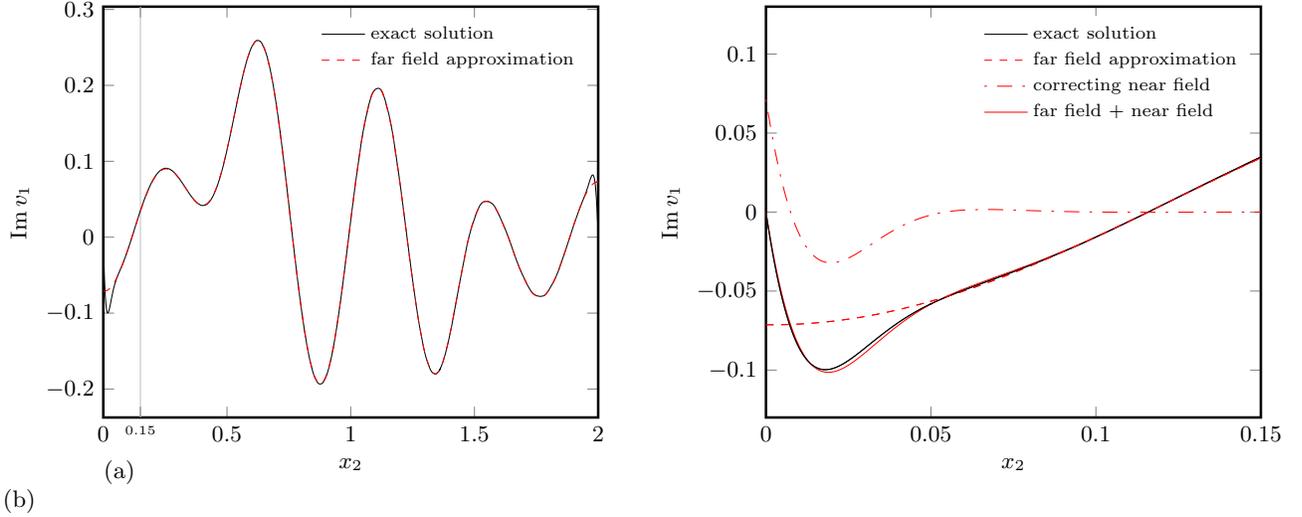
    
    For $\eta = 4\cdot10^{-6}$ and
    $\omega = 15$ we have illustrated the exact pressure and its
    approximation $\apprA{p}{0}$, $\apprA{p}{1}$ and $\apprA{p}{2}$ of
    order~0, 1 and 2, respectively, in the first four subfigures of
    Fig.~\ref{fig:pressure}. %
    The colour scaling in all the four subfigures matches to allow for a
    direct comparison. %
    In this example the approximations of order~0 and order~1 provide a
    coarse field description, where the pressure amplitude is
    overestimated. %
    The approximation of order~2, however, predicts the exact quite
    well. %
    For this example, however, with $\eta = 1.6\cdot10^{-3}$ we have
    illustrated the boundary layer in the tangential velocity component
    in Fig.~\ref{fig:sideview}, both for the exact model and the
    approximation of order~2. %
    The boundary layer thickness is $d_{\mathrm{BL}} =
    \sqrt{{2\eta}/{\omega\rho_0}} = 1.46\cdot10^{-2}$. %
    Here, the approximative far field velocity $\apprA{w}{2}$ and the
    respective near field were computed from the pressure approximation
    $\apprA{p}{2}$. %
    The representation of the velocity is in a side view for $x_1 = 0$,
    for which the first component is tangential to the lower boundary at
    $x_2 = 0$.  The approximate solution is the sum of the far field,
    which does not fulfill a homogeneous Dirichlet boundary condition,
    and a correcting near field. %
    The far field solution approximates the exact one away from the 
    boundary very well, see Fig.~\ref{fig:sideview}(a). 
    In its turn Fig.~\ref{fig:sideview}{b) shows the near field 
        correction and the behaviour of the solutions close to the wall.
        
        
        \begin{figure}[tb]
            {\parbox[t]{1.05\linewidth}{
                    \begin{tikzpicture}[scale=0.96]
                    \pgfplotsset{ legend style={
                            at={(0.99,0.13)},
                            font=\footnotesize, anchor=east} %
                    } %
                    \begin{loglogaxis}[%
                    xlabel=$\sqrt{\eta}$,
                    ylabel=modelling error,
                    xmin=1e-3,xmax=1e-1,
                    ymin=1e-5,ymax=1e1,
                    legend style={draw=none} 
                    ] 
                    \addplot+[only marks] table [x = eps, y = Order0] {RelModError.dat};
                    \addlegendentry{Order $N=0$};
                    \addplot+[only marks] table [x = eps, y = Order1] {RelModError.dat};
                    \addlegendentry{Order $N=1$};
                    \addplot+[only marks, mark=diamond*, violet, mark options={fill=violet},mark size=3pt] 
                    table [x = eps, y = Order2] {RelModError.dat};
                    \addlegendentry{Order $N=2$};
                    \addplot [blue]
                    table [
                    x=eps, 
                    y={create col/linear regression={
                            y=Order0, 
                            variance list={1,1,1,1,1,1,100,100,100,100,100}
                    }}
                    ] {RelModError.dat}
                    coordinate [pos=0.2] (A) 
                    coordinate [pos=0.4] (B)
                    ;
                    \draw (B) -| (A)  
                    node [pos=0.75,anchor=east] {1.0} 
                    ;
                    \addplot [red]
                    table [
                    x=eps, 
                    y={create col/linear regression={
                            y=Order1,
                            variance list={1,1,1,1,1,1,1,100,100,100,100}
                    }}
                    ] {RelModError.dat}
                    coordinate [pos=0.3] (A) 
                    coordinate [pos=0.5] (B)
                    ;
                    \xdef\slope{\pgfplotstableregressiona} 
                    \draw (B) -| (A)  
                    node [pos=0.75,anchor=east] {2.0} 
                    ;
                    \addplot [violet]
                    table [
                    x=eps, 
                    y={create col/linear regression={
                            y=Order2,
                            variance list={100,100,100,100,100,1,1,1,1,100,100}
                    }}
                    ] {RelModError.dat}
                    coordinate [pos=0.5] (A) 
                    coordinate [pos=0.7] (B)
                    ;
                    \xdef\slope{\pgfplotstableregressiona} 
                    \draw (B) -| (A)  
                    node [pos=0.75,anchor=east] {\num[round-mode=figures]{\slope}} 
                    ;
                    \end{loglogaxis}
                    \end{tikzpicture}
                    \hspace{-22em}(a)\hspace{21em}
                    \begin{tikzpicture}[scale=0.96]
                    \pgfplotsset{ legend style={
                            at={(0.99,0.13)},
                            font=\footnotesize, anchor=east} %
                    } %
                    \begin{loglogaxis}[%
                    xlabel=$\sqrt{\eta}$,
                    ylabel=modelling error,
                    xmin=1e-3,xmax=1e-1,
                    ymin=1e-5,ymax=1e1,
                    legend style={draw=none} 
                    ] 
                    \addplot+[only marks,red, mark options={fill=red}] table [x = eps, y = Order1] {RelModErrorEV.dat};
                    \addlegendentry{Order $N=1$};
                    \addplot+[only marks, mark=diamond*, violet, mark options={fill=violet},mark size=3pt] 
                    table [x = eps, y = Order2] {RelModErrorEV.dat};
                    \addlegendentry{Order $N=2$};
                    \addplot [red]
                    table [
                    x=eps, 
                    y={create col/linear regression={
                            y=Order1,
                            variance list={1,1,1,1,1,1,1,100,100,100,100}
                    }}
                    ] {RelModErrorEV.dat}
                    coordinate [pos=0.3] (A) 
                    coordinate [pos=0.5] (B)
                    ;
                    \xdef\slope{\pgfplotstableregressiona} 
                    \draw (B) -| (A)  
                    node [pos=0.75,anchor=east] {1.0} 
                    ;
                    \addplot [violet]
                    table [
                    x=eps, 
                    y={create col/linear regression={
                            y=Order2,
                            variance list={100,1,1,1,1,1,1,100,100,100,100}
                    }}
                    ] {RelModErrorEV.dat}
                    coordinate [pos=0.5] (A) 
                    coordinate [pos=0.7] (B)
                    ;
                    \xdef\slope{\pgfplotstableregressiona} 
                    \draw (B) -| (A)  
                    node [pos=0.75,anchor=east] {\num[round-mode=figures]{\slope}} 
                    ;
                    \end{loglogaxis}
                    \end{tikzpicture}
                    \hspace{-22em}(b)\hspace{22em}
            }}
            \captionsetup{width=\textwidth}
            \caption[Error]{The relative modelling error 
                $\|p-q_{\appr,N}\|_{H^1(\Omega)}/\|p\|_{H^1(\Omega)}+
                \|\vv-\apprA{\vv}{N}\|_{H(\Div,\Omega)}/\|\vv\|_{H(\Div,\Omega)}$
                \mbox{for $N=0,1,2$} w.r.t. square root of viscosity for 
                (a) a dimensionless frequency value $\omega=15$ and
                (b) an eigenfrequency $\omega=\sqrt{20}\,\pi$.}
            \label{fig:error}
        \end{figure}
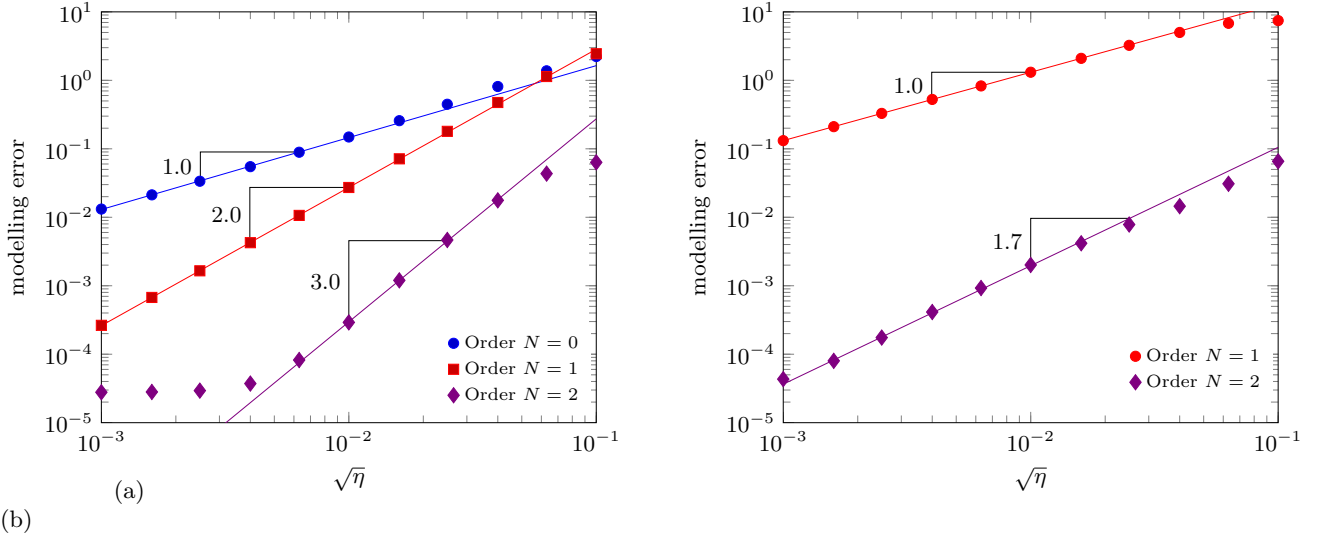
        
        To analyse the modelling error in dependence of the viscosity, and
        hence $\eps$, we have performed numerical simulations on the simple
        rectangular torus domain $\Omega = [0,1]\times[0,1]$ (\ie, without
        the hole of the previous problem),
        for which the left and right
        sides are again identified with each other. %
        The other parameters are identical to those of the previous problem.
        The studied frequency $\omega = 15$ is not a Neumann eigenfrequency
        of $-\Delta$, the closest eigenfrequencies are $\sqrt{20}\,\pi\approx 14.05$ and $5\,\pi\approx 15.71$. %
        We compute the error functions on the subdomain ${\Omega_{\delta}} =
        [0,1]\times[0.2,0.8]$,
        which has a distance of {$\delta = 0.2$} to the boundary of
        $\Omega$. %
        This distance is large enough such that in {$\Omega_{\delta}$} for the
        studied viscosities the contribution of the exponentially decaying
        near fields can be neglected.  %
        In Fig.~\ref{fig:error}(a) we have shown the {relative} modelling error %
        $$\|p -
        \apprA{p}{N}\|_{H^1(\Omega_{\delta})}/\|p\|_{H^1(\Omega_{\delta})} + \|\vv -
        \apprA{\vv}{N}\|_{H(\Div,\Omega_{\delta})}/\|\vv\|_{H(\Div,\Omega_{\delta})}$$
        for the approximative
        solutions of order~0, 1 and 2 in dependence of the (square root of
        the) viscosity. %
        We observe linear convergence in $\sqrt{\eta}$ for the approximative
        solution of order~0, quadratic convergence for that of order~1 and
        convergence of order~3 for the approximative solution of order~2. %
        These results verify that the estimates in Theorem~\ref{lem:error} are
        sharp. %
        The error is computed on the above mesh with polynomial degree 14
        and included indeed a small discretisation error which becomes visible
        for small viscosities ($\sqrt{\eta} < 5\cdot10^{-3}$) and the
        approximative model of order~2.
        
        The theoretical estimates are for non-resonant frequencies and the
        constants may blow up if the frequency tends to a resonant one, \ie, a
        Neumann eigenfrequency of $-\Delta$. %
        The eigenfrequencies for the studied example are %
        ${\omega_{k,m}} = \pi\sqrt{k^2+4m^2}$, for {$k\in\IN,m \in
            \IN_0$}. %
        In addition we analyse the modelling error in dependence of the
        viscosity for an eigenfrequency value $\omega_0{=\omega_{2,2} =
            \omega_{4,1}} = \sqrt{20}\,\pi$, see Fig.~\ref{fig:error}(b). %
        The convergence in this case looses in order, \ie, linear convergence
        in $\sqrt{\eta}$ for the approximative solution of order~1,
        convergence of order $1.7$ for order~2 and the approximative solution
        of order~0 explodes and is not represented in the picture.
        
        Furthermore, we analyse the modelling errors of the three
        approximative solutions in dependence of the frequency for the
        rectangular domain and $\eta = 1.6\cdot10^{-3}$, see
        Fig.~\ref{fig:omega}. %
        The approximate solution of order~0 and so the modelling error blows
        up close to the eigenfrequencies.
        However, the approximate solution of order~1 blows up {only} close
        to the eigenfrequency values ${\omega_{k,0}} = k\pi$ for $k\in\IN$.
        That could be explained by the fact {that for $m = 0$ in this
            example} the velocity and so its divergence is constant in $x_1$}
    and the additional term in the boundary condition of order~1
    disappears. %
    {In this case,} the order~1 approximation at that frequencies
    becomes {identical to} 
    that of order~0. %
    Conversely, the {error of the} approximate solution of order~2, {due to the additional term in the domain},
    always stays {lower than $3\cdot10^{-2}$} and, as it was shown
    earlier, converges w.r.t. viscosity even at the resonance. %
    Yet, in this work we will leave that sentence without a proof and the
    numerical results are presented for illustration reason only.
    
    
    Note, that the above simulation corresponds for dimensionful
    quantities for example to %
    a rectangular domain of size $4\,\mathrm{cm}\times 8\,\mathrm{cm}$,
    where the hole has a diameter of $1.2\,\mathrm{cm}$, a frequency
    $\omega = 5.146\,\mathrm{kHz}$, %
    a speed of sound in air $c = 343\,\mathrm{m}/\mathrm{s}$, %
    a mean density of air $\rho_0 = 1.2\,\mathrm{kg}/\mathrm{m}^3$. %
    Then, a dynamic viscosity of air $\eta = 17.1\,\mathrm{mPa\,s}$
    corresponds to a dimensionless viscosity of $1.04\cdot10^{-6}$
    (dimensionless value of $\sqrt{\eta}$ would be $1.02\cdot10^{-3}$),
    which is close to the lowest viscosity value studied in the above
    experiments.

    \begin{figure}[tb]
        {\parbox[t]{1.05\linewidth}{
                \begin{tikzpicture}[scale=1.1]
                \pgfplotsset{ 
                    legend cell align=left,
                    legend style={
                        font=\footnotesize, anchor=north east} %
                } %
                \begin{semilogyaxis}[%
                xlabel=$\omega$,
                ylabel=modelling error,
                width=0.6\textwidth,
                height=0.4\textwidth,
                xmin=2,xmax=17,
                ymin=1e-5,ymax=1e2,
                legend style={draw=none}, 
                legend pos= south east,
                line width=0.75pt,
                extra x ticks={3.14,6.28,7.02,8.88,9.42477796076937,11.3271733991390,
                    12.5663706143824,12.9531183434152,14.0496294621022,15.7079632696904},  
                extra x tick style={ 
                    xticklabels={},
                    xmajorgrids=true  }
                ]
                \addplot+[mark=none,blue] table [x = omega, y = Order0] {ModErrorOmega.dat};
                \addlegendentry{Order $N=0$};
                \addplot+[mark=none,red] table [x = omega, y = Order1] {ModErrorOmega.dat};
                \addlegendentry{Order $N=1$};
                \addplot+[mark=none, violet] table [x = omega, y = Order2] {ModErrorOmega.dat};
                \addlegendentry{Order $N=2$};
                \end{semilogyaxis}
                \end{tikzpicture}
                \hspace{-30em}(a)\hspace{29em}
                \begin{tikzpicture}[scale=1.1]
                \pgfplotsset{ 
                    legend cell align=left,
                    legend style={
                        font=\footnotesize, anchor=north east} %
                } %
                \begin{semilogyaxis}[%
                xlabel=$\omega$,
                ylabel=modelling error,
                width=0.3\textwidth,
                height=0.4\textwidth,
                line width=0.75pt,
                xmin=13.8496,xmax=14.2496,
                xtick={13.9,14,14.1,14.2},
                ymin=1e-5,ymax=1e2,
                extra x ticks={14.0496},  
                extra x tick style={    
                    xticklabels={}, 
                    xmajorgrids=true            
                }
                ]
                \addplot+[mark=none,blue] table [x = omega, y = Order0] {ModErrorOmega.dat};
                \addplot+[mark=none,red] table [x = omega, y = Order1] {ModErrorOmega.dat};
                \addplot+[mark=none, violet] table [x = omega, y = Order2] {ModErrorOmega.dat};
                \end{semilogyaxis}
                \end{tikzpicture}
                \hspace{-13em}(b)\hspace{13em}
        }}
        \captionsetup{width=\textwidth}
        \caption[Omega dependance]{%
            The modelling error
            $\|p-q_{\appr,N}\|_{H^1(\Omega)}/\|p\|_{H^1(\Omega)}+
            \|\vv-\apprA{\vv}{N}\|_{H(\Div,\Omega)}/\|\vv\|_{H(\Div,\Omega)}$
            \mbox{for $N=0,1,2$} w.r.t. dimensionless frequency $\omega$ for
            $\eta=1.6\cdot10^{-3}$.  }
        \label{fig:omega} 
    \end{figure}
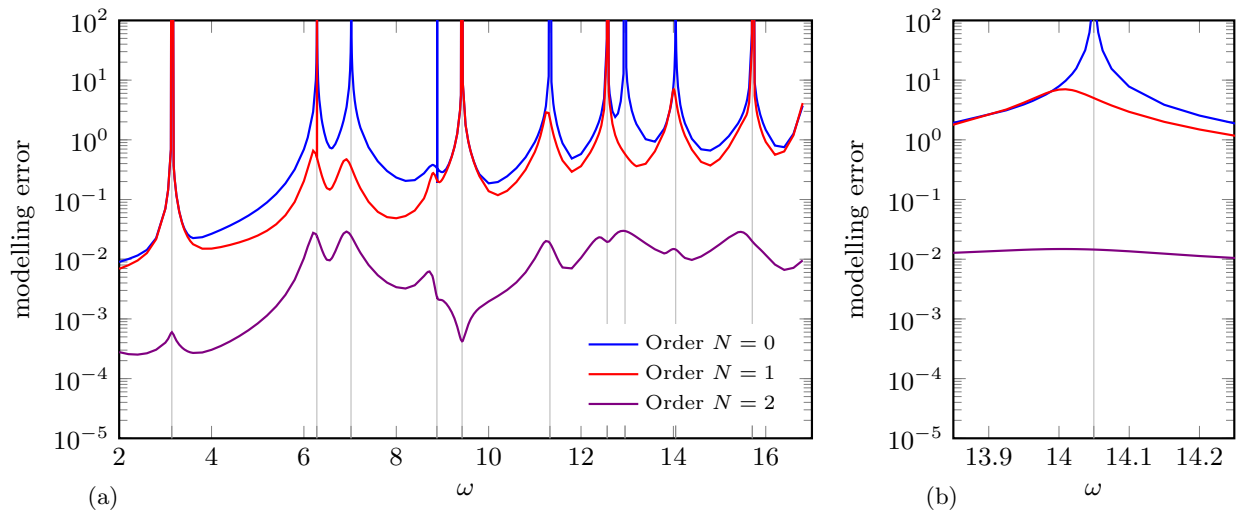

    \section{Conclusion}
    
    In this article the acoustic wave propagation in viscous gases inside
    a bounded two-dimensional domain has been studied as a solution of the
    compressible linearised Navier-Stokes equation.  In frequency domain
    the governing equations are decoupled in equations for the velocity
    and pressure, where the pressure equation lacks boundary conditions.
    The velocity exhibits a boundary layer on rigid walls, whose extend
    scales with the square root of the viscosity and the finite element
    discretisation requires a heavy mesh refinement in the neighbourhood
    of the wall. %
    Using the technique of multiscale expansion for small viscosities
    impedance boundary conditions for velocity and pressure are derived up
    to second order. %
    The derivation and presented analysis is based on a previous work by
    the authors~\cite{Schmidt.Thoens.Joly:2014}, where the complete
    asymptotic expansion of velocity and pressure has been derived. %
    It has be shown that the velocity is represented as a sum of a far
    field expansion, which does not exhibits a boundary layer, and a
    correcting near field expansion close to the wall. %
    For the pressure, which does not exhibit a boundary layer, there is
    only a far field expansion and a near field expansion is absent.
    
    Using boundary conditions for the pressure presented in this work and
    respective partial differential equations pressure approximations are
    defined independently of respective velocities. %
    The zero-th order condition is the well-known Neumann boundary
    condition for rigid walls, and the conditions of first or second order
    take into account absorption inside the boundary layer. %
    The velocity boundary condition is for a far field approximation,
    whose finite element discretisation does not need a special mesh
    refinement close to walls. %
    Here a boundary layer contribution depending on the far field velocity
    can be added to obtain an overall highly accurate description of the
    velocity. %
    The derivation of the boundary conditions for either pressure or
    velocity include curvature effects, where the curvature becomes
    present in the boundary conditions of order 2. %
    
    The approximative models including impedance boundary conditions are
    justified by a stability and error analysis. %
    The results of the numerical experiments have been provided to
    illustrate the stability and error estimates. %
    Although, throughout the article the frequency is assumed to be not an
    eigenfrequency of the limit problem for vanishing viscosity, we show
    by numerical computations that the second order model provides accurate
    approximations for all frequencies and the first order model except
    some of the above mentioned eigenfrequencies.  %
    This results give a foundation for future studies for the case of
    resonances of the limit problem in bounded domains.

\ifx\undefined\allcaps\def\allcaps#1{#1}\fi\def\cprime{$'$}

\end{document}